\documentclass[12pt,leqno]{article}

\usepackage[shortlabels]{enumitem}

\usepackage[all]{xy}			       
\usepackage{amsmath}	
\usepackage{amsthm}				
\usepackage{amsfonts}				
\usepackage{bussproofs}		    
\usepackage{cancel}				    
\usepackage{geometry}				
\usepackage{graphicx}				
\usepackage{hyperref}				
\usepackage[utf8]{inputenc}		

\usepackage{mathrsfs}				
\usepackage{MnSymbol}			
\usepackage{tikz}				        
\usetikzlibrary{tikzmark}
\usepackage{url}				
\usepackage{xcolor}				
\usepackage{prooftree}
\usepackage{proof}

\usepackage{lineno}
\usepackage{multicol}

\usepackage[numbers]{natbib}

\newcommand{\vv}{\vspace*{-1mm}}

\newcommand{\mfA}{\ensuremath{\mathfrak{A}}}

\newcommand{\cc}{\mathbf{C}}

\newcommand{\dv}{\ensuremath{\dashv\vdash}}

\newcommand{\B}{\ensuremath{\mathcal{B}}}

\newcommand{\sneg}{\ensuremath{{\sim}}}

\newcommand{\qf}{$QLET_{F}^+$}

\newcommand{\qletfp}{$QLET_{F}^+$}

\newcommand{\Lg}{\ensuremath{\mathcal{L}}}

\newcommand{\qcletj}{$Q_{c}LET_{J}$}
\newcommand{\qvletj}{$Q_{v}LET_{J}$}

\newcommand{\noi}{\noindent}

\newcommand{\C}{\ensuremath{\mathcal{C}}}
\newcommand{\D}{\ensuremath{\mathcal{D}}}
\newcommand{\I}{\ensuremath{\mathcal{I}}}

\newcommand{\sig}{\ensuremath{\mathcal{S}}}
\newcommand{\ov}{\ensuremath{\overline}}
\newcommand{\wh}{\ensuremath{\widehat}}
\newcommand{\pii}{\ensuremath{\Pi}}

\label{xLOGICS}{}

\newcommand{\letfp}{\ensuremath{LET_{F}^+}}

\newcommand{\letf}{$LET_{F}$}
\newcommand{\letk}{\ensuremath{LET_{K}}}
\newcommand{\letj}{\ensuremath{LET_{J}}}

\newcommand{\letkp}{\ensuremath{LET_{K}^+}}

\newcommand{\letfm}{\ensuremath{LET_{F}^-}}

\newcommand{\fde}{\ensuremath{FDE}}
\newcommand{\fdeto}{\ensuremath{FDE^{\to}}}

\newcommand{\bo}{\textsf{b}}
\newcommand{\nei}{\textsf{n}}

\newcommand{\nel}{\textit{N4}}

\newcommand{\lets}{\textit{LET}s}

\label{xCOLORS} %

\label{xSYMBOLS}

\newcommand{\cons}{\ensuremath{{\circ}}}
\newcommand{\con}{\ensuremath{{\circ}}}
\newcommand{\incon}{\ensuremath{{\bullet}}}
\newcommand{\inc}{\ensuremath{{\bullet}}}

\newcommand{\defi}{\stackrel{\text{\tiny def}}{=} }

\label{xGENERAL COMMANDS}%

\newcommand{\mh}{\noindent}

\newcommand{\m}{\vspace{1mm}}
\newcommand{\mm}{\vspace{2mm}}
\newcommand{\mmm}{\vspace{3mm}}
\newcommand{\mmmm}{\vspace{4mm}}
\newcommand{\setl}{\setlength\itemsep{-0.2em}}

\newcommand{\bqu}{\begin{quote}} 
\newcommand{\equ}{\end{quote}}
\newcommand{\enr}{\begin{enumerate}[label={(\arabic*)}]}  
\newcommand{\eenr}{\end{enumerate}}

\label{xPORTUGUES}

\theoremstyle{definition}
\newtheorem{theorem}{Theorem}[section]
\newtheorem{lemma}[theorem]{Lemma}
\newtheorem{proposition}[theorem]{Proposition}
\newtheorem{corollary}[theorem]{Corollary}

\newtheorem{definition}[theorem]{Definition}
\newtheorem{remark}[theorem]{Remark}

\usepackage[shortlabels]{enumitem}

\usepackage{geometry}				
\begin{document}

\title{Positive, Negative, and Reliable Information in a First-Order Logic of Evidence and Truth
\thanks{The first and second authors  acknowledge support from the National Council for Scientific and Technological Development (CNPq, Brazil), research grants 307889/2025-4 and 309830/2023-0. The second author also acknowledges support 
from São Paulo Research Foundation (FAPESP, Brazil), thematic project RatioLog, grant 2020/16353-3.}
}

\author{Abilio Rodrigues$^1$ and Marcelo E. Coniglio$^2$ \\
[4mm]
$^1$Department of Philosophy\\
Federal University of Minas Gerais (UFMG) \\ abilio.rodrigues@gmail.com \\
$^2$Institute of Philosophy and the Humanities (IFCH), and\\
Centre for Logic, Epistemology and the History of Science (CLE)\\
University of Campinas (UNICAMP) \\
coniglio@unicamp.br
}
 

\maketitle


\abstract{
\mh In this paper we present the first-order logic \qf, a quantified version of the logic \letfp, introduced in Coniglio and Rodrigues 
({Studia Logica} 112:561–606, 2024).
  \qf\ exhibits several properties that are not always enjoyed by logics equipped with classicality operators -- 
we show that it satisfies the replacement property and admits conjunctive, disjunctive, and prenex normal forms.
Alongside extensions and anti-extensions, as in the previously studied first-order semantics for \lets, we make use here of what we call \con-extensions: given an $n$-ary predicate symbol $P$, the \con-extension of $P$ is the set of $n$-tuples of individuals that satisfy the predicate $\con P$. 
We prove the soundness and completeness of the deductive system of \qf\  with respect to the six-valued first-order semantics.}

\maketitle

  \section{Introduction} \label{sec.intro}

Logics of evidence and truth (\lets) form a family of paracomplete and paraconsistent logics that extend the 
logic of first degree entailment \fde, also known as the  
Belnap-Dunn four-valued logic (see e.g.~\cite{belnap.1977.how,dunn76}) by means of a classicality operator $\con$, governed by inference rules
that recover classical negation for sentences within its scope. These logics were introduced together
with an intuitive interpretation in terms of evidence, which may be either conclusive or
non-conclusive \cite{letj,letf}. Accordingly, a formula $\con A$ is read as saying that the evidence available for $A$,
whether positive or negative, is conclusive.

  Information-based logics are logics designed to process information in the sense of treating a database as a set of premises and drawing conclusions from those premises in a coherent and sensible  manner. 
The idea is that what is transmitted from premises to conclusion is not truth but rather the
availability of information, and that an argument is valid when the information conveyed by the
conclusion is already `contained' in the information conveyed by the premises (see
Belnap~\cite[pp.~35--37]{belnap.1977.how}, and also Wansing et al.~\cite{wans.odin,ShramkoWansing2019NatureEntailment}).

The intuitive interpretation of \fde\ in terms of a computer that deals with possibly inconsistent and
incomplete information, together with its four-valued semantics as proposed by Belnap
\cite{belnap.1977.how}, makes \fde\ the first information-based logic to appear in the literature.
The notion of information underlying the standard interpretation of the four semantic values of \fde\ is
that of information as meaningful data, a notion that allows for false information. This is
the same notion of information that underlies contemporary discussions of informational disorder
(see e.g.~\cite{fallis,fetzer.2004.dis,wardle2017}). A logic suitable for formalizing such
contexts cannot, of course, be explosive, and since information may be lacking with respect to some
topics, such a logic must also be paracomplete (cf.~\cite[p.~46]{belnap.1977.how}).

\lets\ can also be interpreted as information-based logics, and in  this setting 
the intuitive meaning of $\con A$ is that the
information conveyed by $A$, whether positive or negative, is reliable. 
In both cases, when $\con A$ holds, $A$ (as well as any formula composed with $A$ over the sentential connectives) is subjected to classical logic. 
Thus, \lets\ are able to express six scenarios: the four scenarios of \fde, represented here by the semantic values $T_0$ and $F_0$ (respectively only positive and only negative information, without the additional information of reliability), \bo\ and \nei\ (respectively \textit{both}, i.e., inconsistent information, and \textit{none}, i.e., lack of information), together with two additional scenarios corresponding to reliable information, positive and negative, represented by the values $T$ and $F$.


Gurevich \citep{gurevich.1977} and Wansing \citep{wans.93} address the problem of the asymmetry
between positive and negative information in intuitionistic logic -- the proof of an atom $A$ a primitive notion, whereas the refutation $\neg A$ of $A$ is not, for it is typically defined as $A\to\bot$. 
 Both Gurevich and Wansing  emphasize the importance
of a formal system in which these two kinds of information are treated on a par, as independent and primitive
(cf.~\cite[pp.~13-14]{wans.93} and \cite[p.~49]{gurevich.1977}).
However, it is in Wansing \citep{wans.93} that this issue is treated
  within the context of a paraconsistent information-based logic. It should be noted 
that although Belnap did not discuss this point explicitly, the idea of treating positive and negative
information independently within an information-based interpretation is already present in 
the four-valued semantics of \fde.


First-order versions of \lets\  have already been investigated \cite{qletf,rod.ant.lu}, 
but in all these cases   the semantics is non-deterministic with respect to a formula $\con A$,  
a common feature of logics equipped with  recovery operators.   
In this paper we introduce the first-order logic \qf, a quantified version of \letfp, originally introduced
in \citep{con.rod.sl}. 
In \qf\  a formula $\con A$ (where $A$ is atomic) 
is   treated as primitive. 
We extend the notion of literal, and call 
a generalized literal formulas $A$, $\neg A$, and $\con A$ (for $A$ atomic). 
 Thus, in addition to   positive and negative information, we have one more primitive notion, namely, 
 reliable information.
A formula $\con  A$, expressing that $A$ is reliable, is subject to two constraints: it cannot hold when both $A$ and $\neg A$ hold, nor when both do not hold.
Evidently, in both cases, there is no reliable information about $A$. 
Once the semantic values of the sentences of the form $A$, $\neg A$, and $\con A$ for an atom $A$ are given by a valuation, the semantic values of all formulas of the language are deterministically assigned.
The logic \qf\ exhibits several noteworthy features that are not usually found in
logics equipped with recovery operators. 
In particular, in addition to providing a sound and complete six-valued
semantics, we show that \qf\ satisfies the replacement property, admits conjunctive and
disjunctive normal forms, and enjoys a prenex normal form theorem.

Building on the idea of interpreting predicates by means of extensions and anti-extensions, as in earlier
first-order semantics for \lets\ \cite{qletf, rod.ant.lu}, our approach also makes use of what
we call \con-extensions. For an $n$-ary predicate symbol $P$, the \con-extension of $P$ is the set of
$n$-tuples of individuals that satisfy the predicate $\con P$.
In \qf\ reliable information is taken as a primitive notion, together
with positive and negative information. For an atomic formula $A$, $A$ represents positive information
$A$, $\neg A$ represents negative information $A$, and $\con A$ represents reliable information $A$.
These three notions are primitive and mutually independent, except for specific constraints governing
$\con A$.

  This paper is organized as follows. 
In Section~\ref{sec.letfp} we review the sentential logic \letfp. We
present both its bivalued semantics and its six-valued twist-structure semantics, as originally given
in \citep{con.rod.sl}. In Section~\ref{sec.properties.letfp} we investigate the behavior of the
non-classicality operator \incon\ and provide proofs of the replacement property for \letfp, as well as
its conjunctive and disjunctive normal forms.
In Section~\ref{sec.qf}, we introduce a first-order extension of \letfp, which we call \qletfp, together  
with a sound and complete six-valued semantics based on the twist structures of \letfp. In
Section~\ref{sec.prop.qf}, we establish the replacement property and the prenex normal form theorem
for \qf. 

 \section{The logic \letfp} \label{sec.letfp}

%
%
%

 Consider a denumerable set ${\mathcal{V}_0}$  of propositional variables and the set of connectives  
$\mathcal{S}_0 = \{ \land, \lor, \neg,\con \}$. 
The language $\mathcal{L}_0$ of \letfp\ is  the set of formulas  generated by 
${\mathcal{V}_0}$ over $\mathcal{S}_0$. 
Roman capitals $A, B, C,\dots$ will be used as metavariables for the formulas of $\mathcal{L}_0$, 
 while Greek capitals $\Gamma, \Delta, \Sigma, \dots$ will be used as metavariables for sets of formulas. 

\m 
We start by the logic \letfm, introduced in \cite{rod.ant.lu} Section~3 as a  minimal logic of evidence and truth. 
\letfm\  just adds to \fde\ rules that recover classical negation for formulas in the scope of \con.

\begin{definition} \label{def.letfm}  (The logic  \letfm)

\m\mh A natural deduction system 
over $\mathcal{L}_0$ for the logic \letfm\ is given by the following inference rules:

\begin{center}

\mm

 $\infer [{I\land}] {A \land B} {A & B}$ \hspace{1.6cm}
$\infer [{E\land}]  {A} {A \land B}  \hspace{2 mm} \infer []  {B} {A \land B}  $ \hspace{6mm}

\

$\infer [{I\lor}] {A \lor B} {A} \hspace{2 mm} \infer [] {A \lor B} {B}$ \hspace{1cm}
$\infer [{E\lor}] {C} {A \lor B & \infer*{C} {[A]} & \infer*{C} {[B]}}$\hspace{6mm}

\

$\infer [{I\neg {\land}}] {\neg(A \land B)} {\neg A } \hspace{2 mm} \infer [] {\neg(A \land B)} {\neg B}$\hspace{1cm}
$\infer [{E\neg {\land}}] {C} {\neg(A \land B) & \infer*{C} {[\neg A]} & \infer*{C} {[\neg B]}}$\hspace{6mm}

\

$\infer [{I\neg {\lor}}] {\neg(A \lor B)} {\neg A & \neg B}$ \hspace{1.2cm}
$\infer [{E\neg {\lor}}]  {\neg A} {\neg (A \lor B)}  \hspace{2 mm} \infer []  {\neg B} {\neg (A \lor B)}  $ \hspace{1mm}

\ 

$\infer [{DN}] {\neg \neg A} {A}  \hspace{2 mm} \infer [] {A} {\neg \neg A} $

\

$\infer [{EXP^{\circ}}] {B} {\circ A & A & \neg A }$ \hspace{10 mm} 
 $\infer [{PEM^\con}] {A \vee \neg A} {\con A}$\hspace{6mm}

\end{center}

\end{definition}

We now extend the logic \letfm\ with rules that express propagation of classicality, i.e. how the operator \con\ is transmitted 
 from less complex to more complex sentences (cf. \cite{con.rod.sl} Section~3).

 \begin{definition} \label{def.ND.letfp}  

\m\mh A natural deduction system 
over $\mathcal{L}_0$ for the logic \letfp\ is given by adding the following definitions and rules to the logic \letfm:

\mm 

\mh  Definitions: for any formula $A$, let  $A^T \defi \cons A \land A$ and $A^F \defi \cons A \land \neg A$.  

\begin{center}

$\infer [_{I \cons\cons}] {\cons\cons A}{}$ \hspace{12mm} $\infer[_{I \cons\neg}] {\cons \neg A}{\cons A}$  \hspace{12mm} $\infer [_{E \cons\neg}] {\cons A} {\cons \neg A}$

\mmmm

$\infer [_{I\land T}] {(A \land B)^T} {A^T & B^T}$ \hspace{12mm}
$\infer [_{I\land F}] {(A \land B)^F} {A^F} \hspace{2 mm} \infer [] {(A \land B)^F} {B^F}$ \hspace{10mm}

\mmmm

$\infer [_{I\lor T}] {(A \lor B)^T} {A^T} \hspace{2 mm} \infer [] {(A \lor B)^T} {B^T}$ \hspace{10mm}
$\infer [_{I\lor F}] {(A \lor B)^F} {A^F & B^F}$ \hspace{10mm}

\mmmm

%

 $\infer [_{E\land T}]  {A^T} {(A \land B)^T}  \hspace{2 mm} \infer []  {B^T} {(A \land B)^T}  $ \hspace{10mm}
$\infer [_{E\land F}] {C} {(A \land B)^F & \infer*{C} {[A^F]} & \infer*{C} {[B^F]}}$\hspace{6mm}

\mmmm

 $\infer [_{E\lor T}] {C} {(A \lor B)^T & \infer*{C} {[A^T]} & \infer*{C} {[B^T]}}$\hspace{10mm}
$\infer [_{E\lor F}]  {A^F} {(A \lor B)^F}  \hspace{2 mm} \infer []  {B^F} {(A \lor B)^F}  $ \hspace{8mm}

\end{center}

\m\mh A deduction of  $A$ from a set of premises $\Gamma$ in $ND_F$, denoted here by $\Gamma\vdash_{F} A$,
 is defined as usual  for natural deduction systems. 
 
%

\end{definition}

\m 

The rationale of the above rules and their intuitive meaning were presented and discussed in \cite{con.rod.sl}, Section 3.
The idea is to express how both the classical behavior and the operator \con\ are transmitted from less to more complex formulas.
Note that what is recovered is the classical behavior of either $A$ or $\neg A$, once $\con A$ holds.
This does not mean that the formulas $A^T$ and $A^F$ exhibit classical behavior.
Indeed, although $A^T, A^F \vdash B$ holds, $\vdash A^T \lor A^F$ does not, since it may be that neither $A$ nor $\neg A$ (and thus $\con A$) holds -- the corresponding counterexample follows straightforwardly from the semantics below.

\begin{remark} \label{rem:Six}
In~\cite[Proposition~29]{con.rod.sl} it was shown that the logic \letfp\ coincides, up to signature, with  {\bf Six}, the logic preserving degrees of truth associated to involutive Stone algebras, which was introduced in~\cite{can.M.fig.2020}.
\end{remark}

\subsection{Semantics} \label{sec.semantics.letfp}

\subsubsection{Valuation semantics for \letfp} \label{sec.bivalued.letfp}

\begin{definition}  \label{def.bival.letfp} 
A  {\em a bivalued semantics} for \letfp\ is a function $\rho:\mathcal{L}_0 \to \{0,1\}$ satisfying the following properties: 

\enr  
 \item   $\rho(A \land B)=1$ iff $\rho(A)=1$ and $\rho(B)=1$; 
 \item $\rho(A \lor B)=1$ iff $\rho(A)=1$ or $\rho(B)=1$; 
 \item  $\rho(\neg\neg A) = 1$ iff $\rho(A) = 1$; 
 \item $\rho(\neg (A \land B)) = 1$ iff $\rho(\neg A) = 1$ or  $\rho(\neg B) = 1$; 
 \item $\rho(\neg (A \lor B)) = 1$ iff $\rho(\neg A) = 1$ and  $\rho(\neg B) = 1$ ; 
 \item if $\rho(\cons A)=1$, then: $\rho(\neg A)=1$ iff $\rho(A)=0$ .
 \item  $\rho(\cons \cons A)=1$;
 \item $ \rho(\cons \neg A)=\rho(\cons A)$;
 \item  If $\rho(\cons A)=\rho(A)=1$ and $\rho(\cons B)=\rho(B)=1$  then $\rho(\cons(A \land B))=1$;
 \item  If $\rho(\cons A)=\rho(\neg A)=1$  then $\rho(\cons(A \land B))=1$;
 \item If $\rho(\cons B)=\rho(\neg B)=1$  then $\rho(\cons(A \land B))=1$;
 \item If $\rho(\cons(A \land B))=\rho(A)=\rho(B)=1$  then $\rho(\cons A) = \rho(\cons B)=1$;
\item If $\rho(\cons(A \land B))=1$, and either $\rho(\neg A)=1$ or $\rho(\neg B)=1$,  then: \\
  \quad \quad \quad  \quad    either $\rho(\cons A) = \rho(\neg A)=1$ or  $\rho(\cons B)=\rho(\neg B)=1$;
\item If $\rho(\cons A)=\rho(A)=1$  then $\rho(\cons(A \lor B))=1$;
\item If $\rho(\cons B)=\rho(B)=1$  then $\rho(\cons(A \lor B))=1$;
\item If $\rho(\cons A)=\rho(\neg A)=1$ and $\rho(\cons B)=\rho(\neg B)=1$   then $\rho(\cons(A \lor B))=1$;
 \item If $\rho(\cons(A \lor B))=1$, and either $\rho(A)=1$ or $\rho(B)=1$,  then: \\ 
 \quad \quad \quad  \quad  \mbox{either} $\rho(\cons A) = \rho(A)=1$ or  $\rho(\cons B)=\rho(B)=1$;
\item If $\rho(\cons(A \lor B))=1$ and $\rho(A)=\rho(B)=0$  then $\rho(\cons A) = \rho(\cons B)=1$;

\eenr

\m 

\noi The semantical consequence relation $\models_{F}^2$ of \letfp\ with respect to  bivaluations is defined as follows:  
$\Gamma\models_{F}^2 A$ \ if and only if, for every bivaluation $\rho$ for \letfp, if $\rho(B)=1$ for every $B \in \Gamma$, then $\rho(A)=1$. 

\end{definition}
 
\mm  Both the bivalued non-deterministic semantics mentioned above and the six-valued semantics to be presented below were introduced in \cite{con.rod.sl}, where the logics \letkp\ and \letfp\ where introduced. The logic \letfp\ is the $\to$-free fragment of \letkp, and the proofs of soundness and completeness are virtually the same. 
 \letfp\ was investigated in Section~5 of \cite{con.rod.sl}, and we present here both its bivalued and six-valued semantics so that the text remains self-contained.

\subsection{A six-valued semantics for \letfp} \label{sec.six.valued.letfp}

 Let us recall that $\rho$ is a function from sentences of $\mathcal{L}_0$ to $\{ 0,1 \}$.     
 Given a sentence $A$, a bivaluation establishes  the values of $A$, $\neg A$, and $\neg A$. 
Let us also recall  the six scenarios expressed  by \lets\ and how they are represented by bivaluations:

\begin{enumerate} 
\item[] When $\con A$ does not hold, $ \rho(\cons A)=0$: 

\item[] i. $ \rho(A)=1$, $\rho( \neg A)=0$:   only positive information  $A$; \hfill (1,0,0)

\item[] ii. $ \rho(A)=0$, $\rho( \neg A)=1$:   only negative information  $A$;  \hfill (0,1,0)

\item[] iii.   $ \rho(A)=1$, $\rho( \neg A)=1$: contradictory information on $A$;  \hfill  (1,1,0)

\item[] iv.  $ \rho(A)=0$, $\rho( \neg A)=0$: no information at all about $A$; \hfill  (0,0,0)
  \end{enumerate}

\begin{enumerate}

\item[] When $\con A$  holds, $ \rho(\cons A)=1$:

 \item[] v. $ \rho(A)=1$, $\rho( \neg A)=0$: reliable positive information   $A$;  \hfill (1,0,1)

\item[] vi.  $ \rho(A)=0$, $\rho( \neg A)=1$: reliable negative information   $A$.  \hfill  (0,1,1)

 \end{enumerate}

\mh In the right side of each one of the scenarios above there is a triple that  corresponds 
to the values of $A$, $\neg A$, and $\con  A$ in each scenario, that is, $\rho(A),\rho(\neg A), \rho(\con A)$. 
Now, we assign names to each of these scenarios (represented by a triple)  where the name itself is the semantic value expressing the informal interpretation in terms of reliable and unreliable information.

 \begin{multicols}{2}

\enr
\item[] i. $(1,0,0) = T_0$,
\item[] ii. $(0,1,0) = F_0$,
\item[] iii. $(1,1,0) = \bo$,
\item[] iv. $(0,0,0) = \nei$,
\item[] v. $(1,0,1) = T$,
\item[] vi. $(0,1,1) = F$.
\eenr
\end{multicols}

\mh The values $T_0, \bo, \nei, $ and $F_0$ correspond to the four values of \fde, and $T$ and $F$ are 
the new values added to represent the two scenarios added to the four scenarios expressed by \fde\ 
(see \cite[Sect.~1]{con.rod.sl}). 
 The idea is that $T_0$ and $F_0$ are weaker than $T$ and $F$ in the sense that they are not 
 conclusive or reliable, and so a contradiction might be obtained by further investigation. 
  In turn, the semantic values  $T$ and $F$, when assigned to a sentence $A$, indicate that the positive or negative information 
  conveyed by $A$ is reliable (conclusive).

  The bivalued semantics for  \letfp\ is non-deterministic because the semantic value of complex formulas are not 
 functionally determined by the values of its parts: the semantic value of formulas  $\neg p $  and $\con p$ 
 are not functionally determined by the value of $p$. 
  Below, we present a six-valued 
 deterministic semantics for \letfp, with the six semantics values mentioned above,    
obtained by means of a \textit{twist structure}  built upon the bivalued semantics of Definition~\ref{def.bival.letfp} above.

A twist structure for \letfp, defined from the bivalued semantics,     
 is an   algebra  whose domain is formed by triples $(z_1,z_2,z_3)$, called {\it snapshots}, 
 over a Boolean algebra. 
 Each snapshot  represents a three-dimensional  semantic value in which the first coordinate $z_1$ represents the 
 semantic value of a formula $A$ in a given bivaluation $\rho$, and the coordinates $z_2$ and $z_3$ represent  the  
 semantic values of $\neg A$ and $\circ A$, respectively, in this same bivaluation $\rho$.  
 The twist  structure to be presented below for \letfp\  yields a six-valued 
  deterministic matrix in which the set of designated values is formed 
 by the snapshots $z$ such that $z_1=1$, which means that the formula in the position $z_1$ holds, or `is true'. 
   The 2-element Boolean algebra with domain ${\bf 2}=\{0,1\}$ will be denoted by  $\mathcal{B}_2$, and its operations will be denoted by $\sneg$ (Boolean complement),  $\sqcap$ (infimum), and $\sqcup$ (supremum).

\begin{definition} \label{def.twist.letfp} \ (Twist structure for \letfp)

\mh Let ${\bf 2}^3$ be the set of triples $z=(z_1,z_2,z_3)$ over ${\bf 2}$.
The twist structure    $\mathcal{M}= \langle \textsf{B}, \textrm{D}, \mathcal{O}\rangle$  for \letfp\ 
(over the Boolean algebra $\mathcal{B}_2$) 
  is defined as follows:  
 \begin{itemize} \setl  
\item[i.]  The set   $\sf B$ 
is the domain of $\mathcal{M}$, the set of semantic values: 
$$\textsf{B}=\{z \in {\bf 2}^3 \ : \ z_3 \leq z_1 \sqcup z_2 \ \mbox{ and } \ z_1 \sqcap z_2 \sqcap z_3=0  \}, $$

that is, $\textsf{B}  =  \{T, \, T_0, \, \bo, \, \nei, \, F_0, \, F  \}$, where
$$T=(1,0,1)   \ \  T_{0}=(1,0,0)  \ \  \bo=(1,1,0)$$ 
$$ \nei=(0,0,0)  \ \  F_{0}=(0,1,0) \ \  F=(0,1,1)   $$

\item[ii.] The set ${\rm D} \neq \emptyset$, $\textrm{D} \subseteq \sf B$, is the set of designated  semantic values: 
$$\textrm{D}= \{z \in \textsf{B} \ : \ z_1=1 \} =  \{T, \, T_0, \, \bo \},$$
while the set of non-designated semantic values is 
$$\textrm{ND}= \{z \in \textsf{B} \ : \ z_1\neq 1 \} =  \{\nei, \, F_0, \, F  \}$$

\item[iii.] $\mathcal{O}$ is a map that  assigns, to each $n$-ary connective $\#$ of $\sig_0$, a function 
$\tilde{\#} :{\sf B}^n \to  {\sf B}$, 
 defined as follows, for every $z$ and $w$ in $\textsf{B}$: 

\begin{itemize} 

	\item[(1)] $z\,\tilde{\land}\,w = \{u\in \textsf{B} \ : \ u_1=z_1\sqcap w_1  
	\mbox{, } u_2=z_2\sqcup w_2  \mbox{, and } \\ \mbox{ \ \ \ \ \ \ } u_3= (z_1 \sqcap z_3 \sqcap w_1 \sqcap w_3) \sqcup (z_2 \sqcap z_3) \sqcup (w_2 \sqcap w_3)) \}$;

	\item[(2)] $z\,\tilde{\lor}\,w = \{u\in \textsf{B} \ : \ u_1=z_1\sqcup w_1  \mbox{, } u_2=z_2\sqcap w_2  
	\mbox{, and } \\ \mbox{ \ \ \ \ \ \ } u_3= (z_2 \sqcap z_3 \sqcap w_2 \sqcap w_3) \sqcup (z_1 \sqcap z_3) \sqcup (w_1 \sqcap w_3))
	\}$;
\item[(3)] $\tilde{\neg}\,z = \{u\in \textsf{B} \ : \ u_1=z_2  \mbox{, } u_2=z_1  \mbox{ and } u_3=z_3   \}$;
	\item[(4)] $\tilde{\circ}\,z = \{u\in \textsf{B} \ : \ u_1=z_3 \mbox{, } u_2=\sneg z_3 \mbox{ and } u_3=1  \}$.

\end{itemize}

%
%

\end{itemize}

\m\mh The operations (1)-(4)  above can also be presented as follows:
\begin{itemize} 
	\item[(1$'$)] 
$(z_1,z_2,z_3)\,\tilde{\land}\,(w_1,w_2,w_3) =  (z_1\sqcap w_1,z_2\sqcup w_2,(z_1 \sqcap z_3 \sqcap w_1 \sqcap w_3) \sqcup (z_2 \sqcap z_3) \sqcup (w_2 \sqcap w_3))$; 	
	\item[(2$'$)] 
$(z_1,z_2,z_3)\,\tilde{\lor}\,(w_1,w_2,w_3)  =  (z_1\sqcup w_1,z_2\sqcap w_2, (z_2 \sqcap z_3 \sqcap w_2 \sqcap w_3) \sqcup (z_1 \sqcap z_3) \sqcup (w_1 \sqcap w_3))$; 	
	
	\item[(3$'$)] $\tilde{\neg}\,(z_1,z_2,z_3) = (z_2,z_1,z_3 )$;
	\item[(4$'$)] $\tilde{\circ}\,(z_1,z_2,z_3) = (z_3,\sneg z_3 ,1 )$.
\end{itemize}

\end{definition}

 \mh Note that the domain $\textsf{B}$    does not contain the triples $(0,0,1)$ and $(1,1,1)$.  
The restrictions  $z_3 \leq z_1 \sqcup z_2$ and $z_1 \sqcap z_2 \sqcap z_3=0$ in the item i. above 
comply with the  rules 
$PEM^\con$ and $EXP^\con$ and the clause (6) of Definition~\ref{def.bival.letfp}, which 
do not allow bivaluations $\rho$ such that $\rho(A) = \rho(\neg A)=0, \rho(\circ A) =1$, 
  or $\rho(A) = \rho(\neg A)=\rho(\circ A) =   1$. 
 Given the definition of $\sf B$, we can write $v(A) = (v_1, v_2, v_3)$, where 
 $v_i \in \{ 0,1 \}$.

\begin{definition} (Six-valued semantics for \letfp) \label{def.valuation.letfp}

\mh A \textit{valuation} $v$ over the twist structure $\mathcal{M}$ is a function  $v \ :\ \mathcal{L}_0 \to \sf B$ such that:   


\begin{itemize} 

	\item[(v1)] $  v(A \land B) \ = \ v(A) \ \tilde{\land} \ v(B)$;
	\item[(v2)] $  v(A \lor B) \ = \ v(A) \ \tilde{\lor} \ v(B)$;
	\item[(v3)] $  v(\neg A) \ = \ \tilde{\neg} v(A)$;
	\item[(v4)] $  v(\con A) \ = \ \tilde{\con} v(A)$.
	
\end{itemize}

\mh Semantical consequence in \letfp\ with respect to  $\mathcal{M}$, denoted by $\vDash_{F}^6 $, is defined as follows: 
for every set of formulas $\Gamma \cup \{A\} \subseteq \mathcal{L}_0$: 
$ \Gamma\models_{F}^6 A$  if and only if, for every valuation $v$ over $\mathcal{M}$, 
if $v(B) \in \rm D$ for every $B \in \Gamma$, then $v(A) \in \rm D$. 

\mm The six-valued semantics given by the matrix $\mathcal{M}$  for \letfp\ 
can   be described by means of the following tables: 

\mm

\begin{center}
\begin{tabular}{|c||c|c|c|c|c|c|}
\hline
 $\tilde{\wedge}$ & $T$  & $T_0$   & $\bo$ & $\nei$ & $F_0$  & $F$\\[1mm]
 \hline \hline
    $T$    & $T$  & $T_0$ & $\bo$ & $\nei$ & $F_0$ & $F$   \\[1mm] \hline
     $T_0$    & $T_0$  & $T_0$ & $\bo$ & $\nei$ & $F_0$ & $F$  \\[1mm] \hline
     $\bo$    & $\bo$  & $\bo$ & $\bo$ & $F_0$ & $F_0$ & $F$  \\[1mm] \hline
     $\nei$    & $\nei$  & $\nei$ & $F_0$ & $\nei$ & $F_0$ & $F$  \\[1mm] \hline
     $F_0$    & $F_0$  & $F_0$ & $F_0$ & $F_0$ & $F_0$ & $F$  \\[1mm] \hline
     $F$    & $F$  & $F$ & $F$ & $F$ & $F$ & $F$  \\[1mm] \hline
\end{tabular}
\hspace{1cm}
\begin{tabular}{|c||c|} \hline
$\quad$ & $\tilde{\neg}$ \\[1mm]
 \hline \hline
    $T$   & $F$    \\[1mm] \hline
     $T_0$   & $F_0$    \\[1mm] \hline
     $\bo$   &$\bo$    \\[1mm] \hline
     $\nei$   & $\nei$    \\[1mm] \hline
     $F_0$   & $T_0$    \\[1mm] \hline
     $F$   & $T$    \\[1mm] \hline
\end{tabular}

\

\

\begin{tabular}{|c||c|c|c|c|c|c|}
\hline
 $\tilde{\lor}$ & $T$  & $T_0$  & $\bo$ & $\nei$ & $F_0$  & $F$ \\[1mm]
 \hline \hline
    $T$    & $T$  & $T$ & $T$ & $T$ & $T$ & $T$   \\[1mm] \hline
     $T_0$    & $T$  & $T_0$ & $T_0$ & $T_0$ & $T_0$ & $T_0$   \\[1mm] \hline
     $\bo$    & $T$  & $T_0$ & $\bo$ & $T_0$ & $\bo$ & $\bo$   \\[1mm] \hline
     $\nei$    & $T$  & $T_0$ & $T_0$ & $\nei$ & $\nei$ & $\nei$   \\[1mm] \hline
     $F_0$    & $T$  & $T_0$ & $\bo$ & $\nei$ & $F_0$ & $F_0$   \\[1mm] \hline
     $F$    & $T$  & $T_0$ & $\bo$ & $\nei$ & $F_0$ & $F$  \\[1mm] \hline
\end{tabular}
\hspace{1cm}
\begin{tabular}{|c||c|}
\hline
 $\quad$ & $\tilde{\circ}$ \\[1mm]
 \hline \hline
    $T$   & $T$    \\[1mm] \hline
     $T_0$   & $F$     \\[1mm] \hline
     $\bo$   & $F$     \\[1mm] \hline
     $\nei$   & $F$     \\[1mm] \hline
     $F_0$   & $F$     \\[1mm] \hline
     $F$   & $T$    \\[1mm] \hline
\end{tabular}
\end{center}
\end{definition}


 The six-valued semantics of Definition~\ref{def.twist.letfp} is, as expected, 
equivalent to the bivalued semantics of Definition~\ref{def.bival.letfp}. 
This has been proved in \citep{con.rod.sl}, Propositions~16 and~18, for the logic \letkp, 
and the proof is essentially the same for the present case. 
Likewise, the soundness and completeness proof of the deductive system $ND_F$ 
with respect to the bivalued semantics follow the same pattern as the proof for \letk\ 
(see \cite{con.rod.sl}, Theorem~4; recall that \letf\ is the $\to$-free fragment of \letk). 
Therefore:

\begin{theorem}\label{th.letfp.sentential.sound.comp}
$\Gamma \vdash_{F} A$ \ iff \ $\Gamma \models_{F}^6 A$ \ iff \ $\Gamma \models_{F}^2 A$.      
\vv
\begin{proof}
See \cite{con.rod.sl}, Theorems~14, 17, 19, and Section~5.      
\end{proof}
\end{theorem}

\section{Some properties of \letfp} \label{sec.properties.letfp}

A well-known drawback of the logic \fde\ is that it lacks an implication satisfying \emph{modus ponens} and the \emph{deduction theorem}. 
In \letfp, these inferences do not generally hold either.  
Consider implication  defined as
\[
A \to B \;\defi\; \neg A \lor B.
\]
It is straightforward, under the semantics introduced above, to construct counterexamples to the following propositions:
\[   A, \neg A \lor B \vdash B, \]
\[A \vdash B \text{ implies } \vdash \neg A \lor B\]

\mh Nevertheless, this flaw is partially solved in \(\letfp\), since, as expected, for formulas assumed to be classical 
(i.e. reliable), \emph{modus ponens} and the \emph{deduction theorem} hold: 
\[ \con A ,  A, \neg A \lor B \vdash B, \]
\[\con A,A \vdash B \text{ implies } \con A \vdash \neg A \lor B\]
The proofs are straightforward and are left to the reader.

\mm    In the proposition below, we present an alternative deductive system for \(\letfp\) that mirrors the clauses of the bivalued semantics and is based only on the primitive connectives (without defining \(A^T\) and \(A^F\)).

\begin{proposition} \label{prop.ND.con.letfp}   \  

\m\mh The  natural deduction system $ND_F'$ for \letfp\  is obtained by 
 adding to the logic \letfm\ (Definition~\ref{def.letfm}) the following inference rules:

\begin{center}
\begin{small}

$\infer[I\con\land_1] {\cons(A \land B)}{\con A & A & \con B & B}$ \hspace{6mm}
 $\infer[I\con\land_2] {\cons(A \land B)}{\con A & \neg A }$ 
 $\infer[] {\cons(A \land B)}{\con B & \neg B}$ \hspace{6mm}
 
\mmm
$\infer[E\con\land_1]  {\cons A \land \cons B}{\cons(A \land B) & A & B}$ 
\hspace{6mm}
$\infer[E\con\land_2]  {(\cons A \land \neg  A) \lor (\cons B \land \neg  B) }{\cons(A \land B) & \neg A \lor \neg B}$

\mmm

$\infer[I\con\lor_1] {\cons(A \lor B)}{\con A & \neg A & \con B & \neg B }$  \hspace{6mm}
$\infer[I\con\lor_2] {\cons(A \lor B)}{\con A & A }$  
$\infer[] {\cons(A \lor B)}{\con B & B}$ \hspace{6mm}

\mmm

$\infer[E\con\lor_1]  {\cons A \land \cons B}{\cons(A \lor B) & \neg A & \neg B}$ \hspace{6mm}
$\infer[E\con\lor_2]  {(\cons A \land   A) \lor (\cons B \land  B) }{\cons(A \lor B) &  A \lor  B}$

\end{small}
\end{center}


\mm
\mh The systems $ND_F$ of Definition \ref{def.ND.letfp} and $ND_F'$ given above are equivalent.

\begin{proof}
We prove (i) ${I}{\lor}{T}$ and (ii) ${E}{\land}{F}$  in the system $ ND'_F $.

 \begin{small}

\mm\mh (i)

\begin{center}
$\infer[I\land]{(\con A \lor B)\land(A\lor B)=(A\lor B)^T}{\infer[E\land, I\con\lor_1]{\con(A\lor B)}{(\con A \land A)=A^T} 
& \infer[I\lor]{A\lor B}{\infer[E\land]{A}{(\con A \land A)=A^T}}}$

\begin{flushleft}
\mmm\mh (ii)
\end{flushleft}

$ \infer[E\lor,2]{C}{\infer*[\pii]{(\cons A \land \neg  A) \lor (\cons B \land \neg  B)}{\con(A\land B)\land\neg(A\land B)=(A\land B)^F} & \infer*[]{C}{[(\con A\land\neg  A)=A^F]^2} & \infer*[]{C}{[(\con B\land \neg B)=B^F]^2}} $

\end{center}

\mmm\mh Where the derivation $ \pii $ is as follows:

\mmm
\begin{center}

\begin{footnotesize}
$ \infer[E\con\land_2]{(\cons A \land \neg  A) \lor (\cons B \land \neg  B)}
{\infer[E\land]{\con(A\land B)}{\con(A\land B)\land\neg(A\land B)} & \infer[E\neg\land,1]{\neg A\lor\neg B}{\infer[E\land]{\neg(A\land B)}{\con(A\land B)\land\neg(A\land B)} & \infer[I\lor]{\neg A\lor\neg B}{[\neg A]^1} & \infer[I\lor]{\neg A\lor\neg B}{[\neg B]^1}}} $
\end{footnotesize}

\begin{flushleft}
\mm \mh Now, we prove (iii) ${E\con}{\lor_1}$ and (vi) ${I\con}{\land_2}$  in the system $ ND_F $.

\mm 
\mh (iii)
\end{flushleft}
\mm

$ \infer[I\land]{\con A\land \con B}{\infer[E\land]{\con A}{\infer[{E}{\lor}{F}]{(\con A\land\neg A)=A^F} {\infer[I\land]{(\con(A\lor B) \land  \neg(A\lor B))=(A\lor B)^F}
{\con(A\lor B) & \infer[I\neg\lor]{\neg(A\lor B)}{\neg A & \neg B}}} } & 
\infer[E\land]{\con B}{\infer[{E}{\lor}{F}]{(\con B\land \neg B)=B^F} { \infer[I\land]{(\con(A\lor B) \land  \neg(A\lor B))=(A\lor B)^F} 
{\con(A\lor B) & \infer[I\neg\lor]{\neg(A\lor B)}{\neg A & \neg B}} } }} $

\begin{flushleft}

\mm 
\mh (iv)
\end{flushleft}
\mm

$ \infer[E\land]{\con(A\land B)}{\infer[{I\land}{F}]{(A\land B)^F = (\con(A\land B)\land\neg(A\land B)) }
{\infer[I\land]{A^F = (\con A\land  \neg A)}{\con A & \neg A}}} $

\end{center}
\end{small}
\mh The remaining cases are left to the reader.  
\end{proof}
 
\end{proposition}


\begin{proposition} \label{prop.equiv.con} \

\m\mh The following equivalences hold in \letfp:

\begin{enumerate}
\item  ${\cons(A \land B)} \dashv\vdash {(\con A \land \con B \land A \land B ) \lor 
(\con A \land\neg A) \lor (\con B \land\neg B)}$;
\item  ${\cons(A \lor B)} \dashv\vdash {(\con A \land  A) \lor (\con B \land  B) 
\lor (\con A \land \con B \land \neg A \land \neg B )  }$; 

\end{enumerate}

\begin{proof}
Items (1) and (2) can be easily proved in the system $ND'_F$ defined above.  
Indeed,  Item (1) from right to left it follows from rules $I\con\land_1$ and $I\con\land_2$,  
and from left to right  follows from $PEM^\con$, $E\con\land_1$ and $E\con\land_2$.  Item (2) is left to the reader. 
\end{proof}

\end{proposition}

\subsection{The non-classicality operator \inc}  \label{sec.operator.inc}

 A non-classicality operator \incon\ can be defined in  \letfp\    as the negation of classicality. The idea is to express 
 the deductive behavior of non-classical sentences. 
  Such operators were first introduced in the context of \textit{LFI}s as inconsistency operators, see e.g. \cite{carn:marcos:deamo:2000}. In \letfp, \con\ and \incon\ are dual to each other in the sense of \cite{recovery} Section~3 -- roughly speaking, given an inference $A\vdash B$, we swap the positions of premise and conclusion and 
 replace the connectives of $A$ and $B$ with their respective duals, being $\lor$ and $\land$ dual 
 to each other and $\neg$ dual to itself, obtaining $B^d\vdash A^d$ (this will be illustrated below).

\begin{proposition} \ \label{prop.incon.def.rules} 

\m\mh The non-classicality operator \incon\ can be defined in \letfp\ as 
$\incon A \defi \neg\con A$.

\m\mh Derived rules with \incon: 

\center 
\mm\mh \small 

$\infer[Cons] {B}{\con A & \incon A}$ \hspace{9mm} 
$\infer[Comp] {\con A\lor \incon A}{}$ \hspace{9mm}
$\infer[I\incon]{\inc A}{A & \neg A}$ \hspace{9mm}
$\infer[Cases]{A\lor\neg A\lor \inc A}{}$

\mmm \mm

$\infer [I\incon\neg] {\incon \neg A}  {\incon  A}  $ \hspace{12mm}
$\infer [E\incon\neg] {\incon  A}  {\incon \neg A}  $ \hspace{12mm}
$\infer[E\incon\incon]{B}{\incon\incon A}$

\begin{proof} 

Rules $Cons$ and $Comp$ follow from $PEM^\con$, $EXP^\con$, and $I\con\con$.
Below we prove 
(i) $I\incon\neg$, (ii) $Cases$, 
and (iii) $E\incon\incon$. 
The remaining rules are left to the reader. 

\mm
(i) 

$$ \infer [{E\lor,1}] {\incon\neg A} {\infer [{Comp}] {\con\neg A \lor \incon \neg A} {} & \infer [{Cons}] {\incon\neg A} {\infer [E\con\neg] {\cons A} {[\circ \neg A]^1} & \incon A} 
& [\incon\neg A]^1  } $$

\mmm

(ii)

$$\infer[E\lor, 1]{A\lor\neg A \lor\inc A}{\infer[Comp] {\con A\lor\inc A}{} & \infer[I\lor]{A\lor\neg A \lor\inc A}{\infer[PEM^\con]{A\lor\neg A}{[\con A]^1}} & \infer[I\lor]{A\lor\neg A \lor\inc A}{[\inc A]^1}}$$

%

\mmm


(iii)

$$\infer[Cons]{B}{\infer[I\con\neg]{ \con\neg\con A = \con\incon A }{\infer[I\con\con]{\con\con A}{}} & \incon\incon A}$$

\end{proof}

\end{proposition}

\m\mh 
In order to illustrate the duality between \con\ and \incon, as well as between the other connectives, it is left as an exercise to the reader the proof of the  inferences below.  
Note that, in each pair, the inferences on the left and on the right are dual to one another.

\begin{itemize}
  \item[(i)] ${\incon (A \lor B)} \vdash {\incon A \lor \neg A}$ and $\con A \land \neg A \vdash \con (A \land B)$,
  \item[(ii)] ${\incon (A \land B)} \vdash {\incon B \lor B}$ and $\con B \land B \vdash \con (A \lor B)$,
  \item[(iii)] ${\incon (A \land B)} \vdash {\incon A \lor \incon B}$ and $\con A \land \con B \vdash \con (A \lor B)$,
  \item[(iv)] ${\incon A \land \incon B} \vdash {\incon (A \lor B)}$ and $\con (A \land B) \vdash \con A \lor \con B$,
  \item[(v)] {\it Cases} and $EXP^\con$,
  \item[(vi)] {$I\incon$} and $PEM^\con$. 
\end{itemize}

\begin{lemma} \label{lem.val.letfp} \ 

\begin{enumerate}
\item If $\rho(\con(A \ast B))=1, \ast \in \{ \land,\lor \}$, then either $\rho(\con A)=1$ or $\rho(\con B)=1$;
\item $\rho(\con A)=1$ if and only if $\rho(\inc A)=0$. 
\end{enumerate}

\begin{proof} 
 Item 1: for $\ast = \land$ the result follows from clauses (12) and (13) of Definition~\ref{def.bival.letfp}, and 
 analogously, for $\ast = \lor $ it follows from clauses (17) and (18). 
  Item 2:  recall that $\inc A$ is defined as $\neg\con A$ (Proposition~\ref{prop.incon.def.rules}). 
 Now suppose either (i) $\rho(\con A)=1$ and $\rho(\inc A)=1$, or 
 (ii) $\rho(\con A)=0$ and $\rho(\inc  A)=0$. Given that   by clause (7)  $\rho(\con\con A)=1$, 
 in both cases (i) and (ii),  the supposition contradicts clause (6). 
\end{proof}
\end{lemma}

\m 

A  feature of the logic \letfp\ is that, for an atomic sentence $A$, in addition to consider $A$ and $\neg A$ as primitive -- which 
is in line with the idea that  positive and negative  information are primitive and independent of each other --,  
  we consider  $\con A$   as primitive as well, under certain restrictions. 
 This fact provides a sort of `symmetry', or well-behavedness, to the system, and 
   will have a decisive impact on the formulation of the first-order semantics, as we will see  in Section~\ref{sec.qf}. 
We define the notion of \textit{generalized literal}  to also include formulas $\con A$.

\begin{definition} \label{def.literals} \ 

\m\mh Let $A$ be an atomic formula. Then $A$, $\neg A$, and $\con A$ are called \emph{generalized literals}.
\end{definition}

\m 
The measure of formula complexity is defined as usual for \textit{LET}s and \textit{LFI}s. 
Note in the definition below that the complexity of $\con A$ adds 2 to the complexity of $A$, since it is taken to depend on both $A$ and $\neg A$.

\begin{definition} \ \label{def.compl.prop}
\m\mh The complexity $ \cc $ of a formula $A$ of \letfp\  is defined as follows:  

\vv\vv
\begin{itemize}\setl  
\item For $ A $ atomic, $\cc(A) = 1$,
\item ${\cc}(\neg A) = {\cc}(A) + 1  $,  
\item ${\cc}(A\land B) = {\cc}(A) + {\cc}(B) + 1  $,  
\item ${\cc}(A\lor B) = {\cc}(A) + {\cc}(B) + 1  $,  
\item ${\cc}(\con A) = {\cc}(A) + 2  $.  

\end{itemize}

\end{definition}

\subsection{The replacement property} 

In~\cite[Proposition~4.2]{mar.riv.2022} it was shown, by algebraic considerations,  that the logic {\bf Six} enjoys the Replacement property. Since {\bf Six} coincides, up to signature, with \letfp\ (recall Remark~\ref{rem:Six}), the replacement property also holds for $\letfp$. In this section, we will give a direct proof of the replacement property for \letfp, by using proof-theoretic arguments.
This is not a very common result to obtain in logics equipped with recovery connectives.

\begin{theorem} (Replacement property for \letfp) \label{th.repl.sent}

\m\mh 

Let $A$ and $B$ be formulas of \letfp, and let $C(A)$ be a formula containing zero or more occurrences of $A$.  
Let $C(B/A)$ denote the result of replacing one or more occurrences of $A$ by $B$ in $C(A)$.  
If $A \dashv\vdash B$, then $C(A) \dashv\vdash C(B/A)$.

\begin{proof}
The proof is by induction on the complexity of  $C$. 

\begin{enumerate}
\item $C$ is a generalized literal $l$.

Either (i) $A=l$ or (ii) $A$ does not occur in $C$.  In both cases, $C(A) = l $,  

\m If (i), $C(B/A)=B$, $C(A)\dv C(B/A)$ follows from   $l\dv B$.      

\m If (ii), $C(B/A)=C(A)=l$, so $C(A)\dv C(B/A)$.

\mm
\item $C =  D\lor E $

\m $(D\lor E)(A) =  D(A) \lor E(A)$  

(IH) $D(A) \dv D(B/A)$, $E(A) \dv E(B/A)$

$D(A) \lor E(A) {\dv} D(B/A) \lor E(B/A) = (D  \lor E)(B/A)$, by (IH).

\mm
\item $C =  D \land E $.  
Left to the reader.

%
%
%
%
%
%

\mm\item $C = \neg  D $.    

\mm  (i)  $D =  \neg E $. 
(IH) $E(A)  {\dv}   E(B/A)$


 

$\neg\neg E(A)\dv   E(A)$, by DN

$ E(A)  {\dv}   E(B/A) $, by (IH)

$   E(B/A) \dv \neg\neg E(B/A)$, by DN


\mm (ii)  $D  =  E\land F $. 
(IH) $\neg E(A) \dv \neg E(B/A) $, $\neg F(A) \dv \neg F(B/A) $ 

\m
$\neg (E(A) \land F (A)) \dv   \neg  E(A) \lor \neg F(A)$, by De Morgan

$\neg  E(A) \lor \neg F(A){\dv} \neg  E(B/A) \lor \neg F(B/A)$, by (IH) 

   $ \neg  E(B/A) \lor \neg F(B/A) \dv \neg  ( E(B/A) \land  F(B/A))$, by De Morgan


\mm (iii)  $D  =  E\lor  F $. Left to the reader.

\mm (iv)  $D =\con E $. (IH) $\con E(A) \dv \con E(B/A)$.


\m That is, $C=\inc E$. We need the rules \textit{Comp} and \textit{Cons} of 
Proposition~\ref{prop.incon.def.rules}. 

In order to show $\inc E(A) \vdash \inc  E(B/A)$, 
 assume $\inc  E(A)$. From \textit{Comp},  
   $\vdash \con E(B/A)\lor \inc E(B/A)$. 
If  $\inc E(B/A)$ is the case, we have the result. If $\con E(B/A)$ is the case, from (IH) 
we get $\con E(A)$, which together with the assumption $\inc  E(A)$ 
and \textit{Cons}  implies $\inc  E(B/A)$.
 An analogous reasoning proves $\incon E(B/A) \vdash \inc  E(A)$.

\mm  
\item $C = \con D $ 
  
\mm  (i)  $D  = \neg E $. (IH)  $\con  E(A) {\dv} \con  E(B/A)$.
 
$\con\neg E(A) \dv \con E(A) $, by $E\con\neg$ 

$  \con E(A)  \dv \con E(B/A)$, by (IH) 

$\con E(B/A) \dv \con \neg E(B/A)$, by $I\con\neg$

 \mm (ii)  $D  =  E\land F $

\m
(IH) $E(A) \dv E(B/A)$, $F(A) \dv F(B/A)$, $\con E(A) \dv \con E(B/A)$, $\con F(A) \dv \con F(B/A)$. 

\m $\con(E\land F)(A) = \con(E(A) \land F (A))$    

\begin{footnotesize}
    
${\dv} 
(\con E(A) \land \con F(A) \land E(A) \land F(A) ) \lor 
(\con E(A) \land\neg E(A)) \lor 
(\con F(A) \land\neg F(A))$, by Prop.~\ref{prop.equiv.con}.

$ {\dv} {(\con E(B/A) \land \con F(B/A) \land E(B/A) \land F(B/A) ) \lor 
(\con E(B/A) \land\neg E(B/A)) \lor (\con F (B/A)\land\neg F(B/A))} $, by (IH).

${\dv} \con( E(B/A)\land F(B/A) )  = \con(E\land F)(B/A)$, by {Prop.~\ref{prop.equiv.con}}.  
\end{footnotesize}

\mm (iii) $D =  E\lor  F $. Left to the reader.

\mm (iv) $D =\con E $. 
$\con\con E(A) \overset{I\con\con}{\dv} \con\con E(B/A) $

\end{enumerate}
\end{proof}

\end{theorem}

\subsection{Normal forms}

Another property that \letfp\ enjoys is having analogues to the conjunctive normal form and disjunctive normal form theorems, as shown below. We start by some preliminary results.

\begin{proposition} (Bottom and top particles) \label{prop.bottom} \ 

\m\mh Bottom particles can be defined in \letfp\  as  
$\bot \defi \con A\land A \land \neg A$, $\bot \defi \con A \land \incon A$, or $\bot \defi  \incon\incon A$. 

\m\mh Top particles are defined as $\top \defi \con\con A$. 

\begin{proof}
It follows from the rules  $I\con\con $, $ EXP^\con$, $Cons$, and $E\inc\inc$ (Proposition~\ref{prop.incon.def.rules}).
\end{proof}
\end{proposition} 

\m 

\begin{proposition} (Iteration of $\neg$ and $\con$) \ \label{prop.reduced.forms}

\m\mh Let $A$ be a formula of \letfp. 
Any formula $s_1s_2\dots s_n A$ ($ n \geq 0$) 
where $s_i \in \lbrace \neg, \con, \inc \rbrace  $, 
 is equivalent to one of the six formulas: 
 $A$,   \ $\neg A$,  \ $\con A$,  \ $\inc A$,  \ $ \con\con A \ =  \ \top$,  \ $\inc\inc A \ = \ \bot A$. 

\begin{proof} If $ n = 0 $, then  $s_1s_2\dots s_n A = A  $.  
If $ n > 0 $,   then the result is obtained replacement property 
 and  the rules $DN$, $E\con\neg$, and $E\inc\neg$. 
\end{proof}

\end{proposition}

\m
\begin{proposition} \label{prop.equiv.inc} \

\m\mh The following equivalences hold in \letfp:

\begin{enumerate}
\item $\incon (A\land B) \dashv \vdash (\incon A \lor \incon B \lor \neg A \lor\neg B)\land (\incon A \lor A) 
\land (\incon B \lor B) $;
\item $\incon (A\lor B) \dashv \vdash (\incon A \lor \neg A) \land (\incon B \lor \neg B) \land  
(\incon A \lor \incon B \lor  A \lor B) $. 

\end{enumerate}

\begin{proof}
 The result follows from   replacement property (Theorem~\ref{th.repl.sent}) and  
Proposition~\ref{prop.equiv.con}, and by using the De Morgan laws.    
\end{proof}

\end{proposition}

\begin{proposition} (Normal forms) \label{prop.normal.forms}

\m\mh A formula $A$ of \letfp\ is in \textit{disjunctive normal form}  (DNF) 
if and only if: 

\begin{enumerate}
\item $A$ has the form $B_1\lor B_2 \lor \dots \lor B_n$, $n \geq 1$;
\item Each $B_i$ has the form $C_1\land C_2 \land \dots\land  C_n$, $n \geq 1$;
\item Each $C_i$ is either an atom $p$, a negated atom $\neg p$, a formula $\con  p$,  $\incon p$, 
   $\top$, or $\bot$.  
\end{enumerate}

\m\mh A formula $A$ of \letfp\ is in \textit{conjunctive normal form} 
 (CNF) if and only if: 

\begin{enumerate}
\item $A$ has the form $B_1\land B_2 \land \dots \land B_n$, $n \geq 1$;
\item Each $B_i$ has the form $C_1\lor C_2 \lor \dots\lor  C_n$, $n \geq 1$;
\item Each $C_i$ is either an atom $p$, a negated atom $\neg p$, a formula $\con  p$,  $\incon p$, 
 $\top$, or $\bot$.    
\end{enumerate}

\m\mh For every formula $A$ in the language of \letfp, $A$ is logically equivalent to a formula $B$ 
in disjunctive normal form and a formula $B'$ in  conjunctive normal form.

\begin{proof}\ 

\mh Step 1:   Apply replacement and Propositions~\ref{prop.equiv.con} and \ref{prop.equiv.inc} until there is no occurrence of $\neg, \land$, 
or $\lor$ in the scope of \con\ or \incon.  
  
\m\mh    Step 2:   
    Apply de Morgan, double negation, and Proposition~\ref{prop.reduced.forms} pushing negations inward and eliminating 
    iterations of $\neg$, \con,  and \incon,   until 
           in resulting formula   all negations are directly applied to literals. 
 
   \m\mh 
   Step 3:   Apply distributivity until in the resulting formula no disjunction contains a conjunction (in the case of CNF) or 
    no conjunction  contains a disjunction (in the case of DNF). 
    
\end{proof}

\end{proposition}




\section{Adding quantifiers: the logic \qletfp}    \label{sec.qf}

Now we introduce a first-order version of the logic \letfp, dubbed here \qletfp. 
Like the sentential version, the quantified version has independent positive and negative rules.    
The rules for the quantifiers, as well as the semantics, have been obtained by thinking of 
 the quantifiers $ \forall$ and      $ \exists$  as infinite conjunctions 
 and  infinite disjunctions, respectively.   

\subsection{The logic \qletfp} \label{sectQLETF+}
 
The logical vocabulary of \qletfp\ is composed by the operators  
 $  \lor, \land, \neg$, and $\con$, the quantifiers  $ \forall $ and $\exists$,  
   the individual variables
from $\mathcal{V} = \{v_{i}: i \in \mathbb{N}\}$, and parentheses. 
 We  take a first-order language as a pair $\Lg = \langle \C, \mathcal{P}
\rangle$, where  $\C$  is an infinite  set of \textit{individual constants}
and $\mathcal{P}$ is a non-empty set of \textit{predicate letters}. 
Each element $P$ of $\mathcal{P}$ is assumed to have a finite arity. 
We assume here the usual definitions of  the  notions of 
  \emph{term}, \emph{formula}, \emph{bound/free occurrence of a variable},
\emph{sentence} etc.,  but with the proviso that formulas with void quantifiers
are not allowed.  \label{page.void}

Given a first-order language \Lg, we  denote the set of terms
generated by \Lg\ by $Term(\Lg)$, the set of formulas and the set of
sentences generated by \Lg\ are denoted, respectively, by $Form(\Lg)$ and
$Sent(\Lg)$.  $x$, $y$, $z$ will be used as
metavariables ranging over $\mathcal{V}$; $c$, $c_{1}$, $c_{2}$,$\dots$ as
metavariables ranging over $\C$; 
$t$, $t_{1}$, $t_{2}$, $\dots$ as metavariables
ranging over $Term(\Lg)$, 
and $A$, $B$, $C$, $\dots$ as metavariables ranging
over $Form(\Lg)$. 
Given $t, t_{1}, t_{2} \in Term(\Lg)$, we will use the
notation $t(t_{2}/t_{1})$ to denote the result of replacing every occurrence of
$t_{1}$ in $t$ (if any) by $t_{2}$. Similarly, 
$A(c/x)$ will denote the formula
that results by replacing every free occurrence of $x$ in $A$ by $c$.

%

The deductive systems and the semantics of
the logic \qletfp\  will be formulated only in terms of sentences.
This is the reason  we have assumed  that languages  have an infinite number of individual constants,  
otherwise we might be prevented from applying some quantifier rules due to the lack of enough constants.  




\begin{definition}\label{def.ND.qf} (The deductive system $ND_{QF}$ for \qf)

\m\mh Let \Lg\   be a first-order language, $c \in \C$,
	and $A, B, C \in Sent(\Lg)$. The logic \qletfp\ is defined over \Lg\ by
adding  following  rules to the rules of \letfp\ (Definition~\ref{def.ND.letfp}):

\mm

 \m\mh  Definitions: for any formula $A$, let  $A^T \defi \cons A \land A$ and $A^F \defi \cons A \land \neg A$.

	\begin{center}

				\bottomAlignProof
						\AxiomC{$A(c/x)$}
					\RightLabel{$I\forall$}
					\UnaryInfC{$\forall xA$}
				\DisplayProof
			\qquad
				\bottomAlignProof
						\AxiomC{$\forall xA$}
					\RightLabel{$E\forall  $}
					\UnaryInfC{$A(c/x)$}
				\DisplayProof
			\qquad
				\bottomAlignProof
						\AxiomC{$A(c/x)$}
					\RightLabel{$I\exists$}
					\UnaryInfC{$\exists xA$}
				\DisplayProof
			\qquad
				\bottomAlignProof
						\AxiomC{$\exists xA$}
							\AxiomC{$[A(c/x)]$} \noLine
						\UnaryInfC{$\vdots$} \noLine
						\UnaryInfC{$C$}
					\RightLabel{$E\exists$}
					\BinaryInfC{$C$}
				\DisplayProof

			\end{center}

			\m

			\begin{center}

				\bottomAlignProof
						\AxiomC{$\neg A(c/x)$}
					\RightLabel{$I\neg \forall$}
					\UnaryInfC{$\neg \forall xA$}
				\DisplayProof
				\qquad
				\bottomAlignProof
						\AxiomC{$\neg \forall xA$}
							\AxiomC{$[\neg A(c/x)]$} \noLine
						\UnaryInfC{$\vdots$} \noLine
						\UnaryInfC{$C$}
					\RightLabel{$E\neg \forall $}
					\BinaryInfC{$C$}
				\DisplayProof
			\qquad
				\bottomAlignProof
						\AxiomC{$\neg A(c/x)$}
					\RightLabel{$I\neg \exists$}
					\UnaryInfC{$\neg \exists xA$}
				\DisplayProof
				\qquad
				\bottomAlignProof
						\AxiomC{$\neg \exists xA$}
					\RightLabel{$E\neg \exists$}
					\UnaryInfC{$\neg A(c/x)$}
				\DisplayProof

			\end{center}

			\m

			\begin{center}
\bottomAlignProof
						\AxiomC{$\forall x( B \lor Ax)$}
					\RightLabel{$CD$}
					\UnaryInfC{$B \lor \forall x A$}
				\DisplayProof

\mmmm

$\infer[{I \forall T}] {(\forall x A)^T}{(A (c/x))^T}$  \hspace{4mm} $\infer[{E \forall T}] {(A (c/x))^T} {(\forall x A)^T}$ \hspace{4mm} 
$\infer[{I \exists T}] {(\exists x A)^T}{(A (c/x))^T}$  \hspace{4mm} $\infer[{E \exists T}] {C} { (\exists x A)^T & \infer*{C}{[(A (c/x))^T]} }$

\mmmm \mm

$\infer[{I \forall F}] {(\forall x A)^F}{(A (c/x))^F}$  \hspace{3mm} 
$\infer[{E \forall F}] {C} { (\forall x A)^F & \infer*{C}{[(A (c/x))^F]} }$ \hspace{3mm} 
  $\infer[{I \exists F}] {(\exists x A)^F}{(A (c/x))^F}$  \hspace{3mm}  $\infer[{E \exists F}] {(A (c/x))^F} {(\exists x A)^F}$  

\mmmm \mm

$\infer[{CD'}] {B \lor (\forall x A)^T } {\forall x (B \lor A^T)}$   

 			\end{center}

%

\mm
\mh 

\begin{description}
\item Restrictions: 	
\item  In $I\forall  $ and $I\neg \exists $,  
	$c$ must not occur in $A$, nor in any hypothesis on which $A(c/x)$ (respectively $\neg A(c/x)$) depends. 
%
	
	\item
	In $E\exists $ and
	$E\neg \forall $, $c$ must occur neither in $A$ nor in $C$, nor in any
	hypothesis on which $C$ depends, except $A(c/x)$ (respectively $\neg A(c/x)$).

	\item 
In $I\forall T$ and $I \exists F$,  
	$c$ must not occur in $A$, nor in any hypothesis on which $(A(c/x))^T$ (respectively $(A(c/x))^F$) depends. 
%

	\item 
	In $E\exists T$ and $I \forall F$, $c$ must occur neither in $A$ nor in $C$, nor in any
	hypothesis on which $C$ depends, except $(A(c/x))^T$ (respectively $(A(c/x))^F$). 

\item In  $CD$ and  $ CD'$,  $x$ must  not be free in $B$.

\end{description}

\mm

    Given that $\Gamma \cup \{A\} \subseteq \mathrm{Sent}(\Lg)$, we take a \emph{deduction of} $A$ \emph{from} $\Gamma$ \emph{in} \qletfp\ to be defined in the usual way for natural deduction systems (see~\cite[Ch.~2]{troelstra.vandalen.1988}).
    Here, it is enough   to
	say that a  derivation $\pii$ is a finite tree in which each node is either a premise from $\Gamma$ or 
    is obtained from earlier nodes by one of the rules above, and  the sentence at the root (bottommost node) is the \emph{conclusion} of $\pii$.
The notation  $\Gamma \vdash_{Q{F}} A$ means 
	that there exists a derivation in \qf\ from the premises in $\Gamma$
	and whose conclusion is $A$, and the subscript is ommitted  when there is no
	risk of ambiguity.   
\end{definition}

\begin{proposition} \label{prop.ND.con.qf} 

\m\mh An equivalent system, call it $ND'_{QF}$, can be defined  in the primitive language (without defining $A^T$ and $A^F$) 
by adding    to the system $ND_F'$ of Definition~\ref{prop.ND.con.letfp} the rules 
 $\forall I$,  $\forall E$, $\exists I$,  $\exists E$, 
$\neg\forall I$,  $\neg\forall E$, $\neg\exists I$,  $\neg\exists E$, 
 (that is, the  $A^T$- and $A^F$-free rules of Definition~\ref{def.ND.qf} above) 
 plus the following rules   (in ${CD^\con}$, $x$ is not free in $B$):

\mmm
\center 
$\infer[\con\forall I] {\con\forall x B  }{\forall x (B \land\con B )}$
  $\infer {\con\forall x B  }{\exists x (\neg B \land\con B )}$
 \hspace{10mm}
  $\infer[\con\forall E] {\forall x (B \land\con B ) \lor \exists x (\neg B \land\con B )}{\con\forall x B  }$

\mmm

$\infer[\con\exists I] {\con\exists x B  }{\exists x (B \land\con B )}$
  $\infer {\con\exists x B  }{\forall x (\neg B \land\con B )}$
 \hspace{10mm}
  $\infer[\con\exists E] {\exists x (B \land\con B ) \lor \forall x (\neg B \land\con B )}{\con\exists x B  }$

\mmm

\hspace{10mm}
$\infer[_{CD^\con}] {B \lor ( \con(\forall x A) \land \forall x A )}{\forall x (B \lor (\con A\land A))}$ 

\mm
\begin{proof} We prove below the rules (i) $E\con\forall$ and  (ii) $I\con\forall$ 
in the system $ND_{QF}$. 

\mmm
\begin{scriptsize}
(i)
\center \mm
$\infer[{E\lor,2}] {\forall x (Bx\land\con Bx) \lor \exists x (\neg Bx\land\con Bx)}{ {\infer[{PEM^\con}] {\forall x B \lor \neg\forall x B}{\con\forall x B}} & {\infer[I\lor] {\forall x (B\land\con B) \lor \exists x (\neg B\land\con B)}
{\infer[I\forall] {\forall x (B\land\con B)} {\infer[E\forall T] {B(c/x)\land \con B(c/x)=B^T}  {\infer[I\land]{(\con\forall x B \land\forall x B )=(\forall x B)^T}{\con\forall x B & [\forall x B]^2 }}}
}} & {\infer[E\forall F,1] {\forall x (B\land\con B) \lor \exists x (\neg B\land\con B)}{\infer[I\land] {\con\forall x B \land\neg\forall x B = (\forall x B)^F} { \con\forall x B & [\neg\forall x B]^2} & {\infer[I\lor]{\forall x (B\land\con B) \lor \exists x (\neg B\land\con B)}{\infer[I\exists]{\exists x (\neg B\land\con B)}{[(\neg B \land \con B)=B^F]^1}}}}}}$

\mmm
\begin{flushleft}
(ii)
\end{flushleft}

 \mm
\begin{center}
$\infer[E\exists,1] {\con\forall x B}{\exists x (\neg B\land\con B)  &  
\infer[E\land] {\con\forall x B}{\infer[I\forall F] {\con\forall x B \land \neg \forall x B = (\forall x B)^T}{[(\neg B \land\con B)=B^F]^1}} }$  \hspace{6mm} 
$\infer[E\land]{\con\forall x B}
{\infer[I\forall T]{(\forall x B)^T=(\con\forall x B \land \forall x B)}
{\infer[E\forall]{(B\land\con B)=B^T}{\forall x (B\land\con B)}}}$
\end{center}
\end{scriptsize}

\mmmm

Now, we prove  the rules (iii) $E\exists T$ and (iv) $E\exists F$  
in the system $ND'_{QF}$. 

\mmm
\begin{scriptsize}
(iii)
\center


\mm

$\infer[E\exists,2]{C}{\infer[E\lor, 1]{\exists x (B\land\con B)}{\infer[E\con\exists]{\exists x (B\land\con B)\lor \forall x (\neg B\land\con B)}{\infer[E\land]{\con\exists x B}{(\con\exists x B \land  \exists x B) = (\exists x B)^T}} 
& [\exists x (B\land\con B)]^1 & \infer[EXP^\con]{\exists x (B\land\con B)}
{\infer[{E\forall, E\exists }]{\bot}{ (\exists x B)^T = (\con\exists x B \land  \exists x B) & [\forall x (\neg B \land \con B)]^1}}} & \infer*{C}{[(B\land\con B)=B^T]^2}}$

\mmm
\begin{flushleft}
(iv)
\end{flushleft}
\center
\mm

$\infer[E\forall]{(\neg B\land \con B)(c/x)=B(c/x)^T }{\infer[E\lor, 1]{\forall x (\neg B\land \con B) }{ \infer[E\con\exists] {\exists x (  B\land \con B) \lor \forall x (\neg B\land \con B)}{\infer[E\land]{\con\exists x B}{(\con\exists x B\land \neg\exists x B)=(\exists x B)^F}} 
& [\forall x (\neg B\land \con B]^1) 
& \infer[EXP^\con]{\forall x (\neg B\land \con B)}{\infer[{E\neg\exists, E\exists}]{\bot}{[\exists x (B\land\con B)]^1 & \infer[E\land]{\neg\exists x B}{(\con\exists x B\land \neg\exists x B)=(\exists x B)^F} }}}}$

\end{scriptsize}
\mmm
The remaining rules are left to the reader. 

\end{proof}
\end{proposition}

 \subsection{Semantics of \qf}   \label{sec.smtcs.qt}

Below we introduce structures for the logic \qf, 
which are first-order structures for quantified 
logics of evidence and truth with propagation of classicality and constant domains. 
These structures will be defined in terms of (six-valued) valuations, but they can be equivalently defined 
 in terms of bivaluations. They extend the  idea of interpreting predicates in terms of extensions 
 and anti-extensions. In addition to the latter, a $n$-ary predicate  $P$ also has a \con-extension, 
 which is the set of  \textit{n}-tuples  of individuals of the domain that satisfy $\con P$.

 \begin{definition}  (\qf-structures) \label{def.structure.qf}
 
 \m\mh  Consider the  matrix ${\cal M}$ for \letfp\ (Definition~\ref{def.twist.letfp}), 
 and let $\mathcal{L} = \langle \C, \mathcal{P}  \rangle$ be a   first-order language.  
 A  {\it first-order structure}  $\mfA$ over ${\cal M}$ and $\mathcal{L}$ is a pair $\langle \mathcal{D},  \mathcal{I} \rangle$ 
 such that $\mathcal{D}$ is a nonempty set (the \textit{domain} of  $\mfA$) and $\I$ is an \textit{interpretation function} that assigns: 
\begin{enumerate}
\item[1.] For each  constant $c \in \mathcal{C}$, an element $\I(c)$ of $\D$;

\item[2.] For each predicate   $P \in \mathcal{P}$, of arity $n$, an interpretation function  $\I(P): \D^n \to  \sf B$.
\end{enumerate}
\end{definition}

\

Recall from Definition~\ref{def.twist.letfp}, which introduces the twist structure $\mathcal{M}$ for \letfp, that each semantic value  in $\sf B$ is a triple $z=(z_1, z_2, z_3)$, where $z_i \in \{0,1\}$, and $0$ and $1$ can be read as \textit{does not hold} and \textit{holds}, respectively. Using this notation, the following can be easily proven.

\begin{proposition}  \label{prop.alt.struc.qf} \ 

\m\mh 
(I) For each predicate   $P \in \mathcal{P}$, of arity $n$, each interpretation $\I(P)$ induces a triple $\langle  P_+, P_-, P_\con \rangle$, where $P_+ ,  P_-, $ and $ P_\con $ are subsets of $\D^n$, such that:

\begin{itemize} \setl 

\item[(a)] $\langle c_1 ,\dots,c_n  \rangle \in P_{+} $ if and only if  $(I(P)(c_1,\ldots,c_n))_1 =1$, 

\item[(b)] $\langle c_1 ,\dots,c_n  \rangle \in P_{-} $ if and only if   $(I(P)(c_1,\ldots,c_n))_2 =1$, 

\item[(c)] $\langle c_1 ,\dots,c_n  \rangle \in P_{\con} $ if and only if  $(I(P)(c_1,\ldots,c_n))_3 =1$.   

\end{itemize}

 \mm 
\mh (II)   
 Structures for \qf\ can also be equivalently defined by replacing the 
 clause (2) of Definition~\ref{def.structure.qf} with the following clause:

\begin{itemize}
\setl 
\item[2$'$.] 
 For each predicate   $P \in \mathcal{P}$, of arity $n$,  $\I(P)$ is a triple $\langle  P_+, P_-, P_\con \rangle$ 
 such that: 
 
 (i) $P_+ \cup P_- \cup P_\con  \subseteq \D^n $, 
 
(ii) For all $a\in\D$, if $a\in P_\con$,  then  $a\in P_+$ if and only if $a\notin P_-$.  

\end{itemize}

\end{proposition}

The sets    $P_+ $  and  $ P_- $ are the extension and the anti-extension of $P$ and contain the $n$-tuples 
of individuals  that satisfy, respectively, $P$ and $\neg P$. The set  $ P_\con $, called the \con-extension of $P$, 
 contains the $n$-tuples of  individuals that satisfy $\con P$. It is immediate to see that\\

$
 \begin{array}{ll}
    P_+ =   & \big\{\vec c \in \D^n \ : \ I(P)(\vec c) \in \{T,T_0,\bo\}\big\}, \\
     P_- =   & \big\{\vec c \in \D^n \ : \ I(P)(\vec c) \in \{F,F_0,\bo\}\big\}, \\
      P_{\con} =   & \big\{\vec c \in \D^n \ : \ I(P)(\vec c) \in \{T,F\}\big\}. \\
 \end{array}
$

\

Observe that no conditions are imposed ensuring $P_+ \cup P_- = \D^n$ 
or excluding the possibility that $P_+ \cap P_- \not = \emptyset$.
Item 2$'$/(ii) corresponds to the clause (6) 
of Definition~\ref{def.bival.letfp} and does not allow precisely the triples in which 
 $P_+ \cap P_- \cap P_\con$ is empty, and the ones in which $P_\cons$ is not contained in $P_+ \cup P_-$. These scenarios would correspond     to the snapshots $(1,1,1)$ and $(0,0,1)$, which indeed does not belong to $\sf B$.  

\m
Bivaluations for \qf-structures will be defined below (Definition~\ref{def.bival.qf}), 
but it  is clear from this discussion that, being $\rho$ a bivaluation, for each predicate   $P \in \mathcal{P}$, of arity $n$, 
the following holds: 

\begin{itemize} \setl


\item[(a$'$)]  $\rho(P(c_1,\ldots,c_n)) =1$ if and only if $\langle c_1 ,\dots,c_n  \rangle \in P_{+} $, 

\item[(b$'$)]  $\rho(\neg P(c_1,\ldots,c_n)) =1$ if and only if $\langle c_1 ,\dots,c_n  \rangle \in P_{-} $, 

\item[(c$'$)]  $\rho(\con P(c_1,\ldots,c_n)) =1$ if and only if $\langle c_1 ,\dots,c_n  \rangle \in P_{\con} $.  

\end{itemize}

\subsubsection{Valuations induced by structures}

We will adopt here a substitutional interpretation of the quantifiers and, accordingly, we introduce a diagram 
language to specify the semantics.
Given a language $\Lg$ and a structure $\mfA$, the diagram language   is obtained by adding to $\Lg$ a fresh individual constant $\bar{a}$ for each $a$ in the domain of $(\mfA)$.

\begin{definition} (Diagram language) \label{def.diagram}   

 \m\mh Let $\Lg = \langle \C,	\mathcal{P} \rangle$ be first-order language and 
  $\mfA$ be a structure over ${\cal M}$ and $\Lg$. 
	The \emph{diagram language} $\Lg_{\mfA}$ of $\mfA$ is the pair 
	$\langle	\C_{\mfA}, \mathcal{P} \rangle$ such that 
	$\C_{\mfA} = 	\C \cup \{\ov{a}: a \in \D\}$. 
We   use the notation $\wh{\mfA}$ to denote the
	 structure that is just like $\mfA$ except that
	$\ov{a}^{\wh{\mfA}} = a$, for every $a \in \D$.   
\end{definition}
 
Now we   extend the six-valued  semantics of \letfp\ (Definition~\ref{def.twist.letfp}) to first-order.  
From now on, given a  structure $\mfA$, we will  write  $c^\mfA$  and $P^\mfA$ instead of, respectively, 
 $\I(c)$ and $\I(P)$. In addition, if $v: Sen(\Lg_\mfA) \to \sf B$ is a function then we will write $v(A) = (v_1(A),v_2(A),v_3(A))$ for every sentence $A$. That is, $v_i(A)=(v(A))_i$ for $1 \leq i \leq 3$.

\begin{definition} \label{def.6val.qf}  (\qf-valuations induced by structures)

\mm\mh  Let $\mathfrak{A} $ be a  \qf-structure.   
The \textit{valuation} induced by  $\mfA$ 
 over  ${\cal M}$ and $\mathcal{B}_2$ 
is the function $v: Sen(\Lg_\mfA) \to \sf B$ such that:

\enr  
\item $v(P(c_1,\ldots,c_n)) = P^{\mathfrak{A}}({c_1}^{\widehat{\mathfrak{A}}},\ldots,{c_n}^{\widehat{\mathfrak{A}}}) $, 
if $P(c_1,\ldots,c_n)$ is atomic;  

\item  $v(A \land B)  =   \Bigl(v_1(A) \sqcap  v_1(B),  v_2(A) \sqcup  v_2(B),   \\
 \bigl( v_1(A) \sqcap  v_3(A) \sqcap  v_1(B) \sqcap  v_3(B) \bigl) 
\ \sqcup \ \bigl(  v_2(A) \sqcap  v_3(A) \bigl) 
\ \sqcup \ \bigl(  v_2(B) \sqcap  v_3(B) \bigl)  \Bigl)   $;

\item  $v(A \lor B)  =   \Bigl(v_1(A) \sqcup  v_1(B),  v_2(A) \sqcap  v_2(B),   \\
 \bigl( v_2(A) \sqcap  v_3(A) \sqcap  v_2(B) \sqcap  v_3(B) \bigl) 
\ \sqcup \ \bigl(  v_1(A) \sqcap  v_3(A) \bigl) 
 \ \sqcup \ \bigl(  v_1(B) \sqcap  v_3(B) \bigl)  \Bigl)   $;

\item $v(\neg A)  = \Bigl( v_2(A) ,  v_1(A),  v_3(A) \Bigl) $; 
 \item $v(\con A)  = \Bigl(  v_3(A) , \sneg  v_3(A), 1 \Bigl) $;  \label{clause.concon}

\item  $v(\forall x A)  = \Bigl(\bigwedge \{ v_1(A ( \bar{a}/x)) \ : \ a \in {\cal D}\},  
\bigvee \{ v_2(A ( \bar{a}/x)) \ : \ a \in {\cal D} \},  \\
 \bigwedge \{ v_3(A(\bar{a}/x)) \sqcap v_1( A(\bar{a}/x)) \ : \ a \in {\cal D} \}$  
 $\sqcup$
 $\bigvee \{ v_3( A(\bar{a}/x)) \sqcap v_2( A(\bar{a}/x) \ : \ a \in {\cal D}  \} \Bigl)$;   


 \item  $v(\exists x A)  = \Bigl(\bigvee \{ v_1(A ( \bar{a}/x)) \ : \ a \in {\cal D}\},  
\bigwedge \{ v_2(A ( \bar{a}/x)) \ : \ a \in {\cal D} \},$  \\ 
$ \bigwedge \{ v_3(  A(\bar{a}/x)) \sqcap v_2( A(\bar{a}/x)) \ : \ a \in {\cal D} \}$  
 $\sqcup$
 $\bigvee \{ v_3( A(\bar{a}/x)) \sqcap v_1( A(\bar{a}/x) \ : \ a \in {\cal D}  \} \Bigl)$.


\eenr 

\mh A sentence $A$ is	said to \textit{hold} in a given  structure  $\mfA$
	($\mfA \vDash A$) if and only if $v_1(A) = 1$, 
	and a set of sentences	$\Gamma$ is said to hold in $\mfA$ 
	($\mfA \vDash	\Gamma$) if and only if every element of $\Gamma$ holds in $\mfA$.  
	$\Gamma$ is said to have a \textit{model} if it holds in
	some structure.  
	Finally, $A$ is a \emph{semantic consequence} of
	$\Gamma$ ($\Gamma \vDash_{QF}^6 A$) if and only if for every structure $ \mfA$,  
	$\mfA \vDash A$ whenever 	$\mfA \vDash \Gamma$.

\end{definition}


\subsubsection{Bivaluations induced by structures}

As expected, a \qf-structure induces not only a (six-valued) valuation, but also a bivaluation.

	\begin{definition} (\qf-bivaluations induced by structures) \label{def.bival.qf} \

\m\mh  Let $\mathfrak{A} $ be a  \qletfp-structure. 
The \textit{bivaluation}  induced by $\mfA$ is the function $\rho: Sen(\Lg_\mfA) \to \{ 0,1\}$ such that 
$\rho$ satisfies the clauses (1)-(8) of Definition~\ref{def.bival.letfp}
plus the following clauses: 

		\enr 
		\setl
		
		\item[(1$'$)]  $\rho(P(c_1,\dots,c_n)) = 1$ iff $\langle c_1^{\wh{\mfA}},\dots,c_n^{\wh{\mfA}} \rangle \in
			P_{+}^{\mfA}$;    

\item[(2$'$)]  $\rho(\neg P(c_1,\dots,c_n)) = 1$ iff $\langle c_1^{\wh{\mfA}},\dots,c_n^{\wh{\mfA}} \rangle \in
			P_{-}^{\mfA}$;    

		\item[(3$'$)]  $\rho(\con P(c_1,\dots,c_n)) = 1$ iff $\langle c_1^{\wh{\mfA}},\dots,c_n^{\wh{\mfA}} \rangle \in
			P_{\con}^{\mfA}$;    




		  


 


\m
\item[(4$'$)] $\rho((A \land B)^T)=1$ iff $\rho(A^T)=1$ and $\rho(B^T)=1$; 
 \item[(5$'$)] $\rho((A \lor B)^T)=1$ iff $\rho(A^T)=1$ or $\rho(B^T)=1$; 
 \item[(6$'$)] $\rho((A \land B)^F) = 1$ iff $\rho(A^F) = 1$ or  $\rho(B^F) = 1$; 
 \item[(7$'$)] $\rho((A \lor B)^F) = 1$ iff $\rho(A^F) = 1$ and $\rho(B^F) = 1$; 

 \m
		\item[(8$'$)]  $\rho(\forall xB )=1$ iff for 	every $a \in \D$, $\rho(B(\bar{a}/x)) =1$,  

\item[(9$'$)]   $\rho(\exists x B) =1$ iff for some $a \in \D$, $\rho(B(\bar{a}/x)) =1$;

		\item[(10$'$)]   $\rho(\neg \forall x B ) =1$ iff for some $a \in \D$, $\rho(\neg B(\bar{a}/x)) =1$;  
		
\item[(11$'$)]  $\rho(\neg \exists x B) =1$ iff for every $a \in \D$, $\rho(\neg B(\bar{a}/x) ) =1$; 

\m		\item[(12$'$)]   $\rho(\con \forall x B) =1$ iff: \\		
\mbox{\quad\quad} for every $a \in \D$, $\rho(B(\bar{a}/x)) =1$ and $\rho(\con B(\bar{a}/x)) =1$, or \\		
\mbox{\quad\quad} for some $a \in \D$, $\rho(\neg B(\bar{a}/x))=1$ and  $\rho(\con B(\bar{a}/x)) =1$.  

		\item[(13$'$)]  $\rho(\con \exists x B) =1$ iff: \\		
\mbox{\quad\quad} for some $a \in \D$, $\rho( B(\bar{a}/x)) =1$ and  $\rho(\con B(\bar{a}/x))=1 $, or \\
\mbox{\quad\quad} for every $a \in \D$, $\rho(\neg B(\bar{a}/x)) =1$ and $\rho(\con B(\bar{a}/x)) =1$.

\eenr   

\mh A sentence $A$ is	said to \textit{hold} in a given  structure  $\mfA$ with respect to bivaluations 
	($\mfA \vDash_{QF}^2 A$) if and only if $\rho(A) = 1$, 
	and a set of sentences	$\Gamma$ is said to hold in $\mfA$ 
	($\mfA \vDash_{QF}^2	\Gamma$) if and only if every element of $\Gamma$ holds in $\mfA$.  
	Finally,  is a \emph{semantic consequence} of
	$\Gamma$ with respect to bivaluations ($\Gamma \vDash_{QF}^2 A$) if and only if for every structure $ \mfA$,  
	$\mfA \vDash_{QF}^2 A$ whenever $\mfA \vDash_{QF}^2 \Gamma$.
\end{definition}

In Definition~\ref{def.bival.letfp}, clauses~(9)-(18), which correspond to the propagation rules for the operator $\con$, are reproduced verbatim from \cite{con.rod.sl}. These clauses can, however, be simplified into a more intuitive form 
    by means of clauses~(4$'$)-(7$'$) above. 
 The latter  mirror the rules of Definition~\ref{def.ND.letfp} and are similar in form to the clauses for $\lor$ and $\land$ 
 in \fde. 

 \begin{proposition} \  \label{prop.novasclausulas}

\m \mh In the presence of clauses (1)-(8) of Definition~\ref{def.bival.letfp}, it holds that clauses (4$'$)-(7$'$) of Definition~\ref{def.6val.qf} are equivalent to clauses (9)-(18) 
of Definition~\ref{def.bival.letfp}. 
\begin{proof} 
Regarding  conjunction, it is easy to see that clause~(4$'$)/Def.~\ref{def.6val.qf}, from right to left, is equivalent to clause~(9)/Def.~\ref{def.bival.letfp}, and from left to right it is equivalent to clause~(12)/Def. \ref{def.bival.letfp}.  
Similarly, clause~(5$'$)/Def.~\ref{def.6val.qf}, from right to left, is equivalent to clauses (10) and (11)/Def. \ref{def.bival.letfp}, and from left to right it is equivalent to clause~(13)/Def.~\ref{def.bival.letfp}. 
Analogous reasoning applies to disjunction.

\end{proof}
\end{proposition}

\begin{lemma} \  \label{lema.con.forall} 
 
 \mm\mh 
 Let $\mathfrak{A} $ be a \qletfp-structure 
and $v$ the valuation induced by $\mfA$. Given  formulas $ \forall x A$ and $\exists x A$:   
 
\mmm (i)  $v_3(\forall x A)=1 $ if and only if: 
 
\quad  for every $a \in {\cal D}$, $v_3(  A(\bar{a}/x)) = 1 \mbox{ and } v_1( A(\bar{a}/x)) = 1$, or

\quad  for some $a \in {\cal D}$, $v_3(A(\bar{a}/x)) = 1 \mbox{ and } v_2(A(\bar{a}/x)) = 1$.

\mmm

(ii)  $v_3(\exists x A)=1 $ if and only if: 
 
\quad  for some $a \in {\cal D}$, $v_3(A(\bar{a}/x)) = 1 \mbox{ and } v_1(A(\bar{a}/x)) = 1$, or

\quad  for every $a \in {\cal D}$, $v_3(  A(\bar{a}/x)) = 1 \mbox{ and } v_2( A(\bar{a}/x)) = 1$.

 \begin{proof} \
 
\mh  Item (i): from clause 6 of Definition~\ref{def.6val.qf},  
$v_3(\forall x A)=1  $ iff 
either $\bigwedge \{ v_3(A(\bar{a}/x)) \sqcap v_1( A(\bar{a}/x)) \ : \ a \in {\cal D} \}=1 $ 
or $ \bigvee \{ v_3( A(\bar{a}/x)) \sqcap v_2( A(\bar{a}/x) \ : \ a \in {\cal D}  \} = 1$.  
Now, the result follows from the fact that 
 $\bigwedge \{ v_3(A(\bar{a}/x)) \sqcap v_1( A(\bar{a}/x)) \ : \ a \in {\cal D} \} = 1$  
iff for every $a \in {\cal D}$, $v_3(  A(\bar{a}/x)) = 1 \mbox{ and } v_1( A(\bar{a}/x)) = 1$, 
and 
$\bigvee \{ v_3( A(\bar{a}/x)) \sqcap v_2( A(\bar{a}/x) \ : \ a \in {\cal D}  \}=1$ iff for some $a \in {\cal D}$, $v_3(A(\bar{a}/x)) = 1 \mbox{ and } v_2(A(\bar{a}/x)) = 1$.
Item (ii) is left to the reader. 
 \end{proof}
 \end{lemma}

\begin{proposition}  (The bivaluation $\rho_v$ associated to the valuation $v$) \label{prop.rho.induced.v}

\m\mh
Let $\mathfrak{A} $ be a \qletfp-structure over $\Lg$ and ${\cal M}$, and let $v$ be  the  valuation induced by  $\mathfrak{A}$. Then, the
mapping $\rho_v \ : \ Sen(\Lg_\mfA) \to \{0,1\}$ 
given by $\rho_v(A)=v_1(A)$ is the bivaluation  induced by  $\mathfrak{A}$. In addition, it holds that: $\rho_v(A)=1$ \ iff \ $v(A) \in \textrm{D}$, for every sentence $A$.

\begin{proof}
 We have to show that the function $\rho_v$ so defined satisfies the clauses of Definition~\ref{def.bival.qf}. 
  As for the  clauses inherited from Definition~\ref{def.bival.letfp}   (sentential connectives) the proof  is essentially the same as   the proof of Proposition~16 of \cite{con.rod.sl}. 
Given Proposition~\ref{prop.novasclausulas}, the result also holds for clauses~(4$'$)--(7$'$), based as well on the proof in~\cite{con.rod.sl}. It remains to show that $\rho_v$ satisfies clauses~(1$'$)--(3$'$) and (8$'$)--(13$'$) of Definition~\ref{def.bival.qf}.

 \m\mh  Clauses (1$'$), (2$'$), and (3$'$): the result follows from the definition of $\rho_v$ and the Definition~\ref{def.twist.letfp}, given that   
$\rho_v (A) = v_1(A) $,   $\rho_v (\neg A) = v_1(\neg A) = v_2(A)$,  
and $\rho_v (\con A) = v_1(\con A) = v_3(A)$. 

 \m\mh  Clause (8$'$): $v_1(\forall x B) =  \bigwedge \{ v_1(B ( \bar{a}/x)) \ : \ a \in {\cal D} \}$.  
 Therefore, $v_1(\forall x B)=1$ iff for all $a\in D$, $v_1(B ( \bar{a}/x))=1$. 
 
 \m\mh  Clause (9$'$): $v_1(\exists x B) =  \bigvee \{ v_1(B ( \bar{a}/x)) \ : \ a \in {\cal D} \}$.  
 Therefore, $v_1(\forall x B)=1$ iff for some $a\in D$, $v_1(B( \bar{a}/x))=1$. 
 
\m\mh  As for clauses (10$'$) and (11$'$), the proofs are left to the reader.  
For  clauses (12$'$) and (13$'$), the result  follows directly from  Lemma~\ref{lema.con.forall}. 
\end{proof}\end{proposition}

\begin{proposition} 
(The valuation $ v_\rho $ associated to the bivaluation $\rho$) \label{prop.v.induced.rho}
 
\m\mh
Let $\mathfrak{A} $ be a \qletfp-structure over $\Lg$ and ${\cal M}$, and let $\rho$ be the bivaluation induced by $\mathfrak{A}$. then, the mapping  $v_{\rho} \ : \ Sen(\Lg_\mfA) \to \sf B $ given by 
 $ v_\rho(A) = (\rho(A),\rho(\neg A),\rho(\con A) ) $
 is the  valuation induced by  $\mathfrak{A}$. In addition, it holds that:
   $v_\rho(A) \in \textrm{D}$ \ iff \ $\rho_v(A)=1$, for every sentence $A$.

 \begin{proof}  \ 
\m\mh By definition, $v_\rho(A) \in \textrm{D}$ \ iff \ $\rho_v(A)=1$. Now, we have to show that the function $v_\rho$ so defined satisfies the clauses (1)-(7) of Definition~\ref{def.6val.qf}. 
Regarding the clauses (2)-(5) the proof  is essentially the same as   the proof of Proposition~18 of \cite{con.rod.sl}. It remains to be shown that \(v_\rho\) satisfies clauses (1), (6), and (7)  of the Definition~\ref{def.6val.qf}.
 
 \m\mh Clause (1):  for $A$ is atomic, $v_\rho(A) = (\rho(A),\rho(\neg A),\rho(\con A))$, hence   $\rho(A)=1$ iff $v_\rho(A) \in \sf B$. 

\mm\mh Clause (6):   $A = \forall x B$.

\mm\mh By definition, 
$v_\rho(\forall x B) = (\rho(\forall x B),\rho(\neg \forall x B),\rho(\con \forall x B))$. 
By clauses (4$'$), (6$'$), and (8$'$) of Definition~\ref{def.bival.qf} we have: 

\m
\mh (i) $\rho(\forall x B)=1$ iff  for every $a \in \D$, $\rho(B(\bar{a}/x)) =1$;

\m\mh (ii) $\rho(\neg \forall x B ) =1$ iff for some $a \in \D$, $\rho(\neg B(\bar{a}/x)) =1$;  

\m\mh (iii) $\rho(\con \forall x B) =1$ iff: \\		
\mbox{\quad\quad\quad} for every $a \in \D$, $\rho(B(\bar{a}/x)) =1$ and $\rho(\con B(\bar{a}/x)) =1$, or \\		
\mbox{\quad\quad\quad} for some $a \in \D$, $\rho(\neg B(\bar{a}/x))=1$ and  $\rho(\con B(\bar{a}/x)) =1$.  

\mm\mh The conditions (i), (ii), and (iii) above are equivalent to, respectively: 

\m\mh 
(i$'$) $v_1(\forall x B) = \bigwedge \{a \in {\cal D} : v_1(B ( \bar{a}/x))\}$;

\m\mh 
(ii$'$) $v_2(\forall x B) = v_1(\neg \forall x B) = \bigvee \{ a \in {\cal D} : v_2(B ( \bar{a}/x)) \}$;

\m\mh 
(iii$'$) $v_3(\forall x B) = v_1(\con \forall x B) = $ \\ 
$ \mbox{\quad\quad\quad} \bigwedge \{ v_3(B(\bar{a}/x)) \sqcap v_1(B(\bar{a}/x)) \ : \ a \in {\cal D} \}$  
  $\sqcup$
 $\bigvee \{ v_3(B(\bar{a}/x)) \sqcap v_2(B(\bar{a}/x) \ : a \in {\cal D} \}$. 

\mm\mh The proof of clause (7), $A=\exists x B$, is left to the reader. 
\end{proof}  \end{proposition}

From the last two results, it follows that $v = v_{(\rho_v)}$ and $\rho=\rho_{(v_\rho)}$. Moreover:

\begin{proposition} \ 

\mm \mh For every set of formulas \(\Gamma \cup \{A\} \subseteq Sent(\mathcal{L}) \):
$\Gamma\vDash_{QF}^6 A $ if and only if $ \Gamma\vDash_{QF}^2 A$.\begin{proof}
    The result follows directly from Propositions~\ref{prop.rho.induced.v} and \ref{prop.v.induced.rho}. 
\end{proof}    
\end{proposition}

 \subsection{Soundness} (Soundness of \qletfp\ w.r.t.~valuations)

%
%

 \begin{theorem}\label{th.soundness.qf} {(Soundness Theorem)}   

\m\mh Let $\Lg$ be a first-order language and $\Gamma \cup \{A\} \subseteq Sent(\Lg)$. 
If $\Gamma\vdash A$, then $\Gamma \vDash A$. 

	\begin{proof}
Let $\D$ be a derivation of $A$ from $\Gamma$ in \qletfp\
	and let $n$ be the number of nodes in $\D$. 
	If $n = 1$, then  either
	(i) $A \in \Gamma$,  or 
	(ii) $A$ is the result of the rule $I \con$. 
	(i) If $A \in \Gamma$, then $\Gamma \vDash A$, since $\vDash$ is
	reflexive. 
	(ii) If $A$  results from $I \con$, then $A = \cons\cons B$, and from  Definition~\ref{def.6val.qf}, 
	$v_1(\con\con B) = 1$,   therefore	$\Gamma \vDash A$.

	\mm  
	\noi Now suppose that $n > 1$ and that the result holds for
	every derivation $\D'$ with fewer nodes than $\D$. We prove that $\Gamma \vDash A$. 
	Since $n > 1$, $A$ results from an
	application of one of the rules of \qletfp\ other than $I\con$.

		\begin{enumerate}
\setl
 
		\item Let $A = \forall xB$ and suppose that it
		results from an application of rule $\forall I$ to $B(c/x)$. 
		
		Hence, there is a derivation $\D'$ of $B(c/x)$ from $\Gamma$ such that $\D'$ has fewer
		nodes than $\D$.  		By IH, $\Gamma \vDash B(c/x)$.

		Given the restrictions on $\forall I$, $c$
		occurs neither in $B$, nor in any formula of $\Gamma$.
		Let $ \mfA $ be a structure and suppose that 	$\mfA \vDash \Gamma$. 
By IH, $\mfA \vDash B(c/x)$. 

Now, consider a structure $ \mfA'$  that  differs from $ \mfA$ at most in the 
individual assigned to $c$, that is, $c^{\mfA}$. 

Since $c$ does not occur in $\Gamma$, 
for every $C \in \Gamma$, 
 		$v_1^\mfA(C) = 1$ if and only if $v_1^{\mfA'}(C) = 1$. 
 		
 		Hence, $\mfA'\vDash \Gamma$, and by IH, $\mfA' \vDash \B(c/x)$. 

Clearly, for every  structure  $ \mfA'  $  that  differs from $ \mfA $ at most in  $c^{\mfA}$, 
 $\mfA' \vDash \B(c/x)$. 
  Therefore, for every $a \in \D$, $v_1^\mfA(B(\ov{a}/x)) = 1$, and by Definition \ref{def.6val.qf}, $v_1^\mfA(\forall x B) = 1$.  
 
 Hence, $\mfA \vDash A$.
 
\mm
\item Let $A$ be the 	result  of an application of rule $\exists E$ to $\exists x B$. 
		
		Hence, there are derivations $\D'$ of $\exists x B$ from $\Gamma$ 
 and $\D''$ of $A$ from $\Gamma, B(c/x)$		
		such that 	both have fewer nodes than $\D$.  		
		
		By IH, $\Gamma \vDash \exists x B$ and $\Gamma , B(c/x) \vDash A$. 
 
Consider a structure  $  \mfA $ such that $  \mfA  \vDash \Gamma$. 
Thus, $\mfA\vDash\exists x B$, and so there is an individual $a \in \D$ 
such that $\widehat{\mfA } \vDash B(\bar{a}/x)$.

Now consider a structure $\mfA'$ that is just like $\mfA$   except that 
$c^{\mfA'}=a$, and so  $\mfA'$ 
agrees with $\mfA$ on all sentences in which $c$ does not occur. 
Thus, $\mfA'\vDash B(c/x)$   and $\mfA'\vDash \Gamma$, since $c$ does not occur in $\Gamma$. 
Therefore, $\mfA'\vDash A$. 
But because $c$  does not occur in $A$ it  follows that $\mfA \vDash A$.  

	\mm	\item Let $A = (\exists  x B)^F$ and suppose that it
		results from an application of rule $I\exists F$ to $(B(c/x))^F$. 
				Hence, there is a derivation $\D'$ of $(B(c/x))^F$ from $\Gamma$ such that $\D'$ has fewer
		nodes than $\D$. 
		
		
		Given the restrictions on $I\exists F$, $c$ 	occurs neither in $B$,    
		nor in any formula of $\Gamma$. 
		
		
		By (IH), $\Gamma \vDash (B(c/x))^F$. 
		
		Let $\mfA$ be a structure and suppose that
		$\mfA \vDash \Gamma$.

Now, consider a structure  $ \mfA'$  that  differs from $ \mfA$ at most in the 
individual assigned to $c$, that is, $c^{\mfA}$. 

Since $c$ does not occur in $\Gamma$, 
for every $C \in \Gamma$, 
 		$v_1^\mfA(C) = 1$ if and only if $v_1^{\mfA'}(C) = 1$. 
 		
		Therefore, $\mfA' \vDash \Gamma$, and so $\mfA' \vdash  (B(c/x))^F$, 
		 which means that $\mfA' \vdash  \con B(c/x) \land \neg B(c/x) $.

Clearly, for every  structure $ \mfA'$  that  differs from $ \mfA$ at most in  $c^{\mfA}$, 
 $\mfA' \vDash \con B(c/x) \land \neg B(c/x) $. 
  Therefore, for every $a \in \D$, $v_1^\mfA( \con B(\bar{a}/x) \land \neg B(\bar{a}/x) ) = 1$, that is, 
   for every $a \in \D$, $v_3^\mfA(B(\bar{a}/x))=1$ and $ v_2^\mfA( B(\bar{a}/x) ) = 1$. 
  
Therefore, by Lemma \ref{lema.con.forall} and Definition~\ref{def.6val.qf},  $v_3^\mfA(\exists x B)$ and $v_2^\mfA(\exists x B) =1$, 
  that is,  $\mfA \vDash \con\exists x B\land \neg\exists x B$,.  
 Hence, $\mfA \vDash A$.

%
%
%
%

		\mm\item Now  suppose that  $A = (B(c/x))^T$ 
		results from an application of rule $E\forall T$ to $(\forall x B)^T$. 
		
 						Hence, there is a derivation $\D'$ of $(\forall x B)^T$ from $\Gamma$  and 
such that  $\D'$  has fewer
		nodes than $\D$. 

Suppose $\mfA \vDash \Gamma$. By IH, $\mfA \vDash (\forall x B)^T $, which means that 
$\mfA \vDash \con \forall x B \land  \forall x B $. 

 By Definition~\ref{def.6val.qf} and  Lemma \ref{lema.con.forall}, 
 it follows that for every $a \in \D$, $v_1^\mfA(B(\bar{a}/x))=1$ and $v_3^\mfA(B(\bar{a}/x))=1$, 
 in particular,  $v_1^\mfA(B(c/x))=1$ and $v_3^\mfA(B(c/x))=1$,  
 and so  $\mfA\vDash (B(c/x))^T$. 
 Therefore, $\mfA\vDash A$.

\mm\item Let $A = B \lor (\forall x A)^T$ and suppose it results from an application of $CD'$. 
 
 Hence, there is a derivation $\D'$ of $\forall x ( B \lor (A)^T)$ from $\Gamma$, 
 $\D'$ has fewer  nodes than $\D$. 
		
		
		Given the restrictions on $CD'$, $x$ is not free in $B$. 
 		

		Let $\mfA$ be a structure and suppose that
		$\mfA \vDash \Gamma$. 
		 
		By (IH), $\mfA \vDash \forall x (B \lor (A)^T)$, 
		so for every $a \in \D$, $\mfA \vDash  (B \lor (A)^T)(\bar{a}/x)$.  
		But since $x$ is not free in $B$, either   $\mfA \vDash  B $, or   for every $a \in \D$, $\mfA \vDash (A)^T(\bar{a}/x)$. 
 
  If $\mfA \vDash B$, it also follows that $\mfA \vDash  B\lor \forall x (A)^T$.

Suppose $\mfA \nvDash B$. 
	It follows that for every $a\in \D$, $\mfA \vDash (A)^T(\bar{a}/x)$, and so 
$\mfA \vDash  \forall x (A)^T$, hence  $\mfA \vDash  B\lor \forall x (A)^T$.  
Therefore, $\mfA\vDash A$.

	\mm The proof of the remaining cases is left to the reader.
		\end{enumerate}
	\end{proof}
 \end{theorem}

\subsection{Completeness}   \label{sec.compl.qf}


The completeness proof  given below is a Henkin-style argument adapted to the specific features of the logic \qf. It proceeds in two steps.
First, we show that if $\Gamma \nvdash A$, then $\Gamma$ can be extended to an $A$-saturated set $\Delta$, which is  a set maximal with respect to not deriving $A$, such that:  
(a)  $\Delta$ is closed under $\vdash$; (b) $\Delta$ is prime;  and (c) $\Delta$   has \emph{witnesses} for every universal and 
for every existential sentence.
Second, we show that, given a set $\Delta$ satisfying (a)-(c) above, 
there exists a structure $\mfA$ such that, for every sentence $A$, $A \in \Delta$ iff $\mfA \vDash A$.



	\begin{definition} \label{def.henkinset} \ 
	
	\m\mh 
 Let $\Lg$ be a first-order language and a set  
	 $\Delta \subseteq Sent(\Lg)$. 
	 We say that $\Delta$  is a  {\it Henkin set} if the following holds: 
\begin{itemize} \setl 
	\item[] (i) $\Delta \vdash \exists xB$ iff
	$\Delta \vdash B(c/x)$, for some $c \in \C$;   
	
	\item[]  (ii) $\Delta
	\vdash \forall xB$ iff $\Delta \vdash B(c/x)$, for every $c \in
	\C$.
\end{itemize}	
%
%
	
	\mm\mh 
	We say that $\Delta$ is a {\it regular set} if  the following holds: 
 \begin{itemize} \setl 
\item[]	(i) $\Delta$ is {non-trivial}, i.e. for some $A$, $\Delta \nvdash A$; 
	
	\item[] (ii) $\Delta$ is \textit{closed}, i.e.  if
	$\Delta \vdash A$, then $A \in \Delta$; 
	
	\item[] (iii) $\Delta$ is  \textit{disjunctive}, i.e.  if $\Delta \vdash A
	\lor B$, then $\Delta \vdash A$ or $\Delta \vdash B$.  
\end{itemize}
	\end{definition}

	\begin{lemma}\label{lema.semantico} \ 
	
%

	 \m\mh If $\Delta$ is a regular Henkin set, then:

\enr \setl   
		\item  $B \land C \in \Delta$ iff $B \in \Delta$ and $C \in
		\Delta$;

		\item $B \lor C \in \Delta$ iff $B \in \Delta$ or $C \in
		\Delta$;

		\item  $\neg(B \land C) \in \Delta$ iff $\neg B \in \Delta$
		or $\neg C \in \Delta$;

		\item  $\neg(B \lor C) \in \Delta$ iff $\neg B \in \Delta$
		and $\neg C \in \Delta$;

		\item  $\neg \neg B \in \Delta$ iff $B \in \Delta$;

\item $\cons A \in \Delta$  only if $   A \in \Delta$ iff $ \neg A \notin \Delta$;

 \item $\cons\cons A \in \Delta$;
 \item $\cons A \in \Delta$  iff $\cons\neg A \in \Delta$;

		\item  $(B \land C)^T \in \Delta$ iff $B^T \in \Delta$ and $C^T \in
		\Delta$;

		\item  $(B \lor C)^T \in \Delta$ iff $B^T \in \Delta$ or $C^T \in
		\Delta$;
		
		\item  $(B \land C)^F \in \Delta$ iff $B^F \in \Delta$ or $C^F \in
		\Delta$;

		\item  $(B \lor C)^F \in \Delta$ iff $B^F \in \Delta$ and $C^F \in
		\Delta$;

		\item  $\forall xB \in \Delta$ iff for 	every $a \in \D$, $B(\bar{a}/x) \in \Delta$,  

		\item  $\exists xB \in \Delta$ iff for some $a \in \D$, $B(\bar{a}/x) \in \Delta$;

		\item  $\neg \forall xB \in \Delta$ iff for some $a \in \D$, $\neg B(\bar{a}/x) \in \Delta$;  
		
		 \item  $\neg \exists xB \in \Delta$ iff for every $a \in \D$, $\neg B(\bar{a}/x) \in \Delta$; 

		\item  $\con \forall xB \in \Delta$ iff: \\		
\mbox{\quad\quad} for every $a \in \D$, $B(\bar{a}/x) \in \Delta$ and $\con B(\bar{a}/x) \in \Delta$, or \\		
\mbox{\quad\quad} for some $a \in \D$, $\neg B(\bar{a}/x) \in 	\Delta$ and  $\con B(\bar{a}/x) \in \Delta$.  

		\item  $\con \exists xB \in \Delta$ iff: \\		
\mbox{\quad\quad} for some $a \in \D$, $ B(\bar{a}/x) \in 	\Delta$ and  $\con B(\bar{a}/x) \in  \Delta$, or \\
\mbox{\quad\quad} for every $a \in \D$, $\neg B(\bar{a}/x) \in \Delta$ and $\con B(\bar{a}/x) \in \Delta$. 
 
\eenr 	\begin{proof} The proof os this lemma is routine. Items (1) to (16) are  consequences of the assumption
	that $\Delta$ is a regular Henkin set together with the rules of Definition~\ref{def.ND.qf}.   
	Items (17) and (18)  follow straightforward from the alternative rules of 
    Proposition~\ref{prop.ND.con.qf}.    
    \end{proof}
	\end{lemma}

We now proceed with a Lindenbaum construction to show how an Henkin  set $\Delta$ such that $\Delta \nvdash A$ 
can be obtained from a given set $\Gamma$ such that $\Gamma \nvdash A$.



	\begin{lemma} \label{lema.lind}  (Lindenbaum)

	\m\mh Let $\mathcal{L} = \langle \mathcal{C}, \mathcal{P} \rangle$ be a
	first-order language and $\Gamma \cup \{A\} \subseteq Sent(\mathcal{L})$. If $\Gamma \nvdash A$, then there
	is a language $\Lg^{+} = \langle \mathcal{C}^{+}, \mathcal{P} \rangle$ and a regular Henkin set
	$\Delta \subseteq Sent(\Lg^{+})$ such that $\mathcal{C} \subseteq \mathcal{C}^{+}$, $\Gamma
	\subseteq \Delta$, and $\Delta \nvdash A$.
	\end{lemma}

	\begin{proof} 
	Let $\mathcal{C}^{+} = \mathcal{C} \cup \{c_{i}: i \in \mathbb{N}\}$,
that is, $\mathcal{C}^+$ is $\mathcal{C}$ plus a denumerable number of new constants. 	
 Let $B_{0},B_{1},B_{2},\dots$ be a list of the sentences in $Sent(\Lg^{+})$, 
 and $c_{0},c_{1},c_{2},\dots$ a list of the new constants, i.e., the constants in 
 $ \mathcal{C}^+ \backslash \mathcal{C} $.   
  Now, consider the sequences $\Gamma_{n}$ and 
	$A_{n}$,  ${n \in \mathbb{N}}$, defined as follows:

\begin{itemize}
\item  $\Gamma_{0} = \Gamma$ and $A_{0} = A$;

\item $\Gamma_{n+1} = $
 \begin{itemize}
 \item[i.] $\Gamma_n$ if $\Gamma_{n}, B_{n} \vdash A_{n}$;
  \item[ii.] $\Gamma_n \cup \{B_n \}$ if $\Gamma_{n}, B_{n} \nvdash A_{n}$ and $B_n \neq \exists x C$;
 \item[iii.] $\Gamma_{n} \cup \{B_{n}, C(c_{j}/x) \}$ if 
 $\Gamma_{n}, B_{n} \nvdash A_{n}$ and  $B_{n} = \exists xC$; 	
 
 \end{itemize}

\item $A_{n+1} = $
\begin{itemize}
 \item[i.] $A$ if $\Gamma_{n}, B_{n} \nvdash A_{n}$;
  \item[ii.] $A\lor B_n$ if $\Gamma_{n}, B_{n} \vdash A_{n}$ and $B_n \neq \forall x C$;
 \item[iii.] $A\lor B_n \lor C(c_{j}/x)$   if 
 $\Gamma_{n}, B_{n} \vdash A_{n}$ and  $B_{n} = \forall xC$; 	
 \end{itemize}
 \item[] where $c_j$ is the first constant in the list of the new constants that does not occur in 
 $\Gamma_n$, in $B_i$, nor in $A_i$, $0 \leq i \leq n$. 
\end{itemize}

\m 


	\noi Let $\Delta = \bigcup \Gamma_{n}$, $n \in \mathbb{N}$. Clearly, $\Gamma \subseteq \Delta$. 
	We prove that $\Delta$ is a regular Henkin set such that $\Delta \nvdash A$. 

		\begin{itemize}

		\item[1.] For every $n \in \mathbb{N}$, $\Gamma_{n} \nvdash A_{n}$. 
		
\m		 The proof is by
		induction on $n$.  By the initial hypothesis, $\Gamma_{0} \nvdash A_{0}$. 
		
		Suppose that $\Gamma_{n} \nvdash A_{n}$ (IH). We show that $\Gamma_{n+1} \nvdash A_{n+1}$.  
		
		There are two cases: either (i) $\Gamma_{n}, B_{n}
		\vdash A_{n}$ or (ii) $\Gamma_{n}, B_{n} \nvdash A_{n}$.

\m 

(i) $\Gamma_{n}, B_{n} \vdash A_{n}$, so  $\Gamma_{n+1} = \Gamma_{n}$.  
Suppose $\Gamma_{n+1} \vdash A_{n+1} $, that is  $\Gamma_{n} \vdash A_{n+1} $. 

We have two cases: either (a)  $A_{n+1} = A_n \lor B_n$ 
		if $B_n \neq \forall x C$, 
		or (b) 		$A_{n+1}= A_n \lor \forall x C \lor C(c_j/x)$ if $B_n = \forall x C$.

\m (i.a) $\Gamma_n\vdash A_n \lor B_n$.
 The latter, together with 
	$\Gamma_{n}, B_{n} \vdash A_{n}$ (the initial hypothesis), implies $\Gamma_{n} \vdash A_{n}$, 
	which contradicts (IH). 
	
\m (i.b)  $\Gamma_n \vdash  A_n \lor \forall x C \lor C(c_j/x)$. Since $c_j$ is a new constant that 
does not occur in $A_n$, by applying $\lor E$ and  $CD$, we obtain 
		$\Gamma_n \vdash  A_n \lor \forall x C \lor \forall x C $, that is, 
		$\Gamma_n\vdash A_n \lor B_n$. The latter, with the initial hypothesis, 
	implies $\Gamma_{n} \vdash A_{n}$, 
	which contradicts (IH).

\m Therefore, $\Gamma_{n+1} \nvdash A_{n+1}$

	\mm

\m (ii) $\Gamma_{n}, B_{n} \nvdash A_{n}$, so $A_{n+1} = A_n$.   
Suppose $\Gamma_{n+1} \vdash A_{n+1} $, that is  $\Gamma_{n+1} \vdash A_{n}$. 

\m 
We have two cases: either (a)  $\Gamma_{n+1} = \Gamma_n \cup \{B_n\}$ 
		if $B_n \neq \exists x C$, 
		or (b) 	$\Gamma_{n+1} = \Gamma_n \cup \{ \exists x C, C(c_j/x) \}$	
		if 	$B_n = \exists x C$.

\m (ii.a) $\Gamma_n, B_n \vdash A_n$ contradicts the initial hypothesis (ii).

\m 
(ii.b) $\Gamma_{n+1} = \Gamma_n \cup \{ \exists x C, C(c_j/x) \}$. 
Since  $c_j$ does not occur in $\Gamma_n$, $\exists x C$, nor in $A_n$, by $\exists E$, 
 it follows that $\Gamma_n,\exists x C \vdash A_n$, that is,  $\Gamma_n,B_n \vdash A_n$,  
 which contradicts the initial hypotheses (ii).

\m Therefore, $\Gamma_{n+1} \nvdash A_{n+1}$

		\item[2.] For every $n \in \mathbb{N}$, $\Delta \nvdash A_{n}$ (in particular, $\Delta
		\nvdash A$).

		\m Suppose that $\Delta \vdash A_{n}$, for some $n \in \mathbb{N}$. 
 Hence, there is a derivation of $A_n$ from a finite set of undischarged hypotheses $\Delta_0$ and  
 some $ \Gamma_m$ such that $\Gamma_m \subseteq \Delta_0$ and $\Gamma_m\vdash A_n$
 	
		If $m \leq n$, then $\Gamma_{m} \subseteq \Gamma_{n}$, and so $\Gamma_{n}	\vdash A_{n}$. 
		
		If $m > n$, for some sentence $B \in   Sent(\Lg^{+})$,  $A_{m} = A_{n} \vee B_i$,  
		therefore $\Gamma_{m} \vdash A_{m}$. 
		Both cases contradict item (1) above.

		\item[3.] If $\Delta \vdash C$, then $C \in \Delta$. 
		
		Suppose that $\Delta \vdash C$ and that $C \notin \Delta$. 
Since $C \notin \Delta$,  $C$ is some $B_n$ such that $\Gamma_{n}, B_{n} \vdash A_{n}$. 
It follows that  $\Delta, B_{n} \vdash A_{n}$ and $\Delta \vdash B_n$, 
therefore $\Delta \vdash A_{n}$, which contradicts (2)
		above. 

		\item[4.] If $\Delta \vdash C \vee D$,  then $\Delta \vdash C$ or $\Delta \vdash D$:
		Suppose that $\Delta \vdash C \vee D$ and that $\Delta \nvdash C$ and $\Delta \nvdash D$.
For some $m$ and $n$,  $C=B_m$ and $D=B_n$, and since both are not in $\Delta$,   
  $\Gamma_{m}, B_{m} \vdash A_{m}$ and $\Gamma_{n}, B_{n} \vdash	A_{n}$. 
  Now, given that $\Delta \vdash B_m \vee B_n$, either $\Delta \vdash A_m$ or 
  $\Delta \vdash A_n$, and  both cases contradict item (2) above. 
  Therefore, $\Delta \nvdash C$ and $\Delta \nvdash D$.

		\item[5.] $\Delta \vdash \forall xC$ if and only if
		$\Delta \vdash C(c/x)$, for every $c \in \mathcal{C}^{+}$.

		 We prove only that if for every $c \in \mathcal{C}^{+}$, $\Delta \vdash C(c/x)$,  
		 then $\Delta \vdash \forall x C$, 		
		 since the other direction is an immediate consequence of rule $\forall E$.

		 We prove the contrapositive: if $\Delta \nvdash \forall x C$, then 		 
		 for some $c \in \mathcal{C}^{+}$, $\Delta \nvdash C(c/x)$.

		 Suppose that $\Delta \nvdash B_{n}$ and $B_{n} = \forall xC $.    
			By the definition of the  sequence $\Gamma_{n},   n \in \mathbb{N}$, 
			it   follows that $\Gamma_{n}, B_{n} \vdash A_{n}$ and that $A_{n+1} =
		A_{n} \vee B_{n} \vee C(c_{j}/x)$.

		Suppose that $\Delta \vdash C(c_{j}/x)$.
		Thus, $\Delta \vdash A_{n+1}$, which contradicts (2) above. Hence, $\Delta \nvdash
		C(c/x)$ for at least one $c \in \mathcal{C}^{+}$.

		\item[6.] $\Delta \vdash \exists xC$ if only if $\Delta
		\vdash C(c/x)$, for some $c \in \mathcal{C}^{+}$.

	 We prove only that if $\Delta \vdash \exists xC$, 
	 then for some $c \in \mathcal{C}^{+}$, $\Delta \vdash C(c/x)$, 
  since the other direction is an immediate consequence of rule $\exists I$.

Suppose $\Delta \vdash \exists xC$. So for some step $n$ in the construction of the sequence $\Gamma_n$, 
$B_n=\exists x C$ and $\Gamma_n, \exists x C \nvdash A_n$. 
Therefore, $\Gamma_{n+1}= \Gamma_n \cup \{ \exists x C, Cc_j\}$, where $j$ is a fresh constant. 
Since $\Gamma_{n+1}  \subseteq \Delta$, $C_j \in \Delta$, and so $\Delta\vdash C_j$. 
Hence, for some $c \in \mathcal{C}^{+}$, $\Delta \vdash C(c/x)$, 
		\end{itemize}
	\end{proof}

Lemma \ref{lema.canonical} below 
 shows how to obtain a model from a regular Henkin set. As usual, the model is
 defined in terms of a \qf-structure  construed over the very symbols of the language and 
 the derivability relation. 
 As is well-known,  it is precisely this link between syntax and semantics, obtained by means 
 of a Henkin construction, that allows to prove completeness.

	\begin{lemma} \label{lema.canonical} (Canonical model)

\m\mh  	
 Let $\Lg = \langle \mathcal{C}, \mathcal{P} \rangle$ be a first-order language and  
 $\Delta\cup \{A \} \subseteq Sent(\Lg)$. 
 If $\Delta$ is a regular Henkin set such that $\Delta\nvdash A$,  
 then $\Delta$ induces a \textit{canonical structure} $\mfA_\Delta$ such that   $\mfA_\Delta\models B $ iff $B\in\Delta$, for every sentence $B$.

\begin{proof}

Define $\mfA_\Delta = \langle \mathcal{D}, \mathcal{I} \rangle$ as follows: 

\begin{enumerate}
\item $\D = \C$

\item For each constant $c $,  $\I(c)= c$

\item For each  predicate $P$ of arity $n$,  $I(P):\C^n \to \sf B$ is given by

\[I(P)(\vec c)=
                        \left\{\begin{array}{cl}
                                T & \mbox{if $P(\vec c) \in \Delta$, $\neg P(\vec c) \notin \Delta$, $\cons P(\vec c) \in \Delta$}\\[1mm]
                                 T_0 & \mbox{if $P(\vec c) \in \Delta$, $\neg P(\vec c) \notin \Delta$, $\cons P(\vec c) \notin \Delta$}\\[1mm]
                                 \bo & \mbox{if $P(\vec c) \in \Delta$, $\neg P(\vec c) \in \Delta$, $\cons P(\vec c) \notin \Delta$}\\[1mm]
                                 \nei & \mbox{if $P(\vec c) \notin \Delta$, $\neg P(\vec c) \notin \Delta$, $\cons P(\vec c) \notin \Delta$}\\[1mm]
                                 F_0 & \mbox{if $P(\vec c) \notin \Delta$, $\neg P(\vec c) \in \Delta$, $\cons P(\vec c) \notin \Delta$}\\[1mm]
                                 F & \mbox{if $P(\vec c) \notin \Delta$, $\neg P(\vec c) \in \Delta$, $\cons P(\vec c) \in \Delta$}\\[2mm]
                        \end{array}\right.\]

\end{enumerate}
\noindent for every $\vec c \in \C^n$.
By  $EXP^{\circ}$ and  $PEM^\con$, and by the fact that $\Delta$ is a closed non-trivial theory, it follows that the function $I(P)$ is well-defined. Observe that $I(P)$ induces a triple $\langle  P_+, P_-, P_\con \rangle$ such that: \\

$\langle  c_1, \dots, c_n \rangle \in P_+$ \ iff \ $P(c_1, \dots, c_n) \in \Delta$;

$\langle  c_1, \dots, c_n \rangle \in P_-$ \ iff \ $\neg P(c_1, \dots, c_n) \in \Delta$; and

$\langle  c_1, \dots, c_n \rangle \in P_\con$ \ iff \ $\con P(c_1, \dots, c_n) \in \Delta$.\\

\mh  Let $\rho_\Delta$ be the bivaluation defined as $\rho_\Delta(B)=1 $ iff $B\in\Delta$. 
 It remains to be proved that  $\rho_\Delta$ satisfies all clauses of Definition~\ref{def.bival.qf}.
 
\m  To prove this, note that when   $B$ is $P(c_1, \dots, c_n)$, $\neg P(c_1, \dots, c_n)$, or $\con P(c_1, \dots, c_n)$, 
 it follows from the definition of $\mfA_\Delta$  that $\rho_\Delta$ satisfies clauses 1$'$, 2$'$, and 3$'$  
 of Definition~\ref{def.bival.qf}. 
The remaining clauses follow from Lemma~\ref{lema.semantico}.

From Proposition~\ref{prop.v.induced.rho}, $\rho_\Delta$ induces a valuation $v_\Delta$ that satisfies all clauses of  Definition~\ref{def.6val.qf}, such that, for every sentence $B$: $v_\Delta(B) \in \textrm{D}$ if, and only if, $B \in \Delta$.
 \end{proof}
	\end{lemma}

	\begin{theorem} \label{th.completeness.qf}  {(Completeness of \qletfp\ w.r.t.~valuations)} 
	
	\mm\mh If
	$\Gamma \vDash_6 A$, then $\Gamma \vdash_{QF} A$.
    
	\begin{proof} Suppose that $\Gamma \nvdash A$. By Lemma~\ref{lema.lind}, there is a first-order language
	$\Lg^{+} = \langle \mathcal{C}^{+}, \mathcal{P} \rangle$ and a set $\Delta \subseteq
	Sent(\Lg^{+})$ such that $\mathcal{C} \subseteq \mathcal{C}^{+}$, $\Gamma \subseteq \Delta$, and
	$\Delta$ is a regular Henkin set 
	such that  $\Delta \nvdash A$. 
	By Lemma \ref{lema.canonical}, there
	exists a structure  $\mfA_\Delta$   
	such that $\mfA_\Delta  \vDash \Delta$,  but since $A\notin \Delta$, 
	$\mfA_\Delta \nvDash A$. 
		Let $\mfA_\Delta'$ be  the structure obtained by restricting $\mfA_\Delta$ to the language  
		$\Lg$. 
	 	Clearly, for every $B \in Sent(\Lg)$, $\mfA_\Delta' \vDash B$ if, and only if, $\mfA_\Delta \vDash B$.  
		As a result, $\mfA_\Delta' \vDash \Gamma$ (since $\Gamma \subseteq \Delta$) 
	but $\mfA_\Delta' \nvDash A$. Therefore, $\Gamma \nvDash A$.
	\end{proof}	\end{theorem}

Some remarks on the completeness proof are worth making here.
The proof relies on the equivalence between (six-valued) valuations and bivaluations of \qf. 
In Definition~\ref{def.structure.qf}  we defined \qf-structures  in terms of valuations by introducing an interpretation function 
$\mathcal{I}$ that 
determines a value in the set $\sf B$ to each atomic sentence $A$.
In Proposition~\ref{prop.alt.struc.qf}, we showed that each structure induces a triple $\langle P_+, P_-, P_\con \rangle$, namely the extension, the anti-extension, and the \con-extension of an $n$-ary predicate $P$. This triple determines both valuations and bivaluations with respect to $P$.
The canonical model $\mfA_\Delta$ (Lemma~\ref{lema.canonical}) is defined in terms of bivaluations 
($\rho_\Delta(B) = 1 \text{ iff } B \in \Delta$), which, given Lemma~\ref{lema.semantico}, makes it easier to show that the canonical model defines a bivaluation (Definition~\ref{def.bival.qf}).
The proof of Lemma~\ref{lema.canonical} shows that the canonical model is in fact a \qf-structure, as defined by Definition~\ref{def.structure.qf}, since the  bivaluation $\rho_\Delta$, by Proposition~\ref{prop.v.induced.rho} induces the valuation $v_\Delta$ that satisfies Definition~\ref{def.6val.qf}. 
Given soundness, the  result below follows immediately:

\begin{corollary}   \label{corol.comp.sound} \

\mm\mh $ \Gamma\vdash_{QF} A$ if and only if $ \Gamma\vDash_6 A$  if and only if  $ \Gamma\vDash_2 A$.
\begin{proof}
 The result follows from Propositions~\ref{prop.rho.induced.v}, \ref{prop.v.induced.rho}, and Theorems~\ref{th.soundness.qf} and 
 \ref{th.completeness.qf}.  
\end{proof}  

\end{corollary}


\section{On some properties  of \letfp\ and \qf}   \label{sec.prop.qf}

We start this section by presenting definitions and preliminary results to be used in the proof of the replacement property, 
as well as in the prenex normal form theorem for \qf.

\begin{definition} \ \label{def.compl.prop1}

 \m\mh The complexity $ \cc $ of a formula $A$ of \qf\  is defined as follows:



\vv
\begin{itemize}\setl  
\item For $ A $ atomic, $\cc(A) = 1$,
\item ${\cc}(\neg A) = {\cc}(A) + 1  $,  
\item ${\cc}(A\land B) = {\cc}(A) + {\cc}(B) + 1  $,  
\item ${\cc}(A\lor B) = {\cc}(A) + {\cc}(B) + 1  $,  
\item ${\cc}(\con A) = {\cc}(A) + 2  $,
\item ${\cc}(\forall A) = {\cc}(A) + 1  $,
\item ${\cc}(\exists A) = {\cc}(A) + 1  $. 
\end{itemize}

\end{definition}

\begin{proposition} \ \label{prop.equiv.repl.qf} 

\m\mh The following equivalences hold in \qf:

\begin{enumerate}[label=(\arabic*)]
 \setl 
%

\item ${\con\forall x B } \dv {\forall x (B\land\con B) \lor \exists x (\neg B\land\con B)}$
\item ${\con\exists x B } \dv {\exists x (B\land\con B) \lor \forall x (\neg B\land\con B)}$
%

 \item {$\neg \forall x  A \dashv\vdash \exists x \neg  A$}
\item {$\neg\exists x  A\dashv\vdash \forall x\neg   A$}
\item {$\neg \forall x \neg A\dashv\vdash\exists  x  A$}
\item {$  \neg \exists x \neg A \dashv\vdash \forall x A$}

\end{enumerate}

\begin{proof}
Items (1) and (2) follow immediately from Proposition \ref{prop.ND.con.qf}, 
while items (3) to (6) can be easily proved by means of the rules for quantifiers.
\end{proof}
\end{proposition}

\begin{proposition}  \label{prop.equiv.prenex} 

\m\mh The following equivalences hold in \qf:

\begin{enumerate}[label=(\arabic*)]
 \setl 
\item $\forall x (B\lor A) \dashv\vdash B\lor \forall x A $, $x$ is not free in $B$;
\item $\exists x (B\lor A) \dashv\vdash B\lor \exists x A $, $x$ is not free in $B$;
\item $\forall x (B\land A) \dashv\vdash B\land \forall x A $, $x$ is not free in $B$;
\item $\exists x (B\land A) \dashv\vdash B\land \exists x A $, $x$ is not free in $B$.    
\end{enumerate}

\begin{proof}
Item~(1), from left to right, is immediate  from $CD$. The remaining cases are easily proved by means of the quantifier rules of \qf.
\end{proof}
\end{proposition}

\subsection{Replacement \qf}

\begin{theorem} 
(Replacement property) \label{th.replacement.qf}

\m\mh 
Let $A$ and $B$ be formulas of \qf, and let $C(A)$ be a formula containing zero or more occurrences of $A$.
Denote by $C(B/A)$ the formula obtained from $C(A)$ by replacing one or more occurrences of $A$ with $B$.
Then,  $A \dashv\vdash B$ implies that $C(A) \dv  C(B/A)$.

\begin{proof}
The proof is by induction on the complexity $\cc$ of  $C$. We need to add the following cases to the proof of 
 Theorem~\ref{th.repl.sent}. 

\begin{enumerate}

\item $C = \forall x D$. 

\m(IH) $D(c/x)(A) \dv D(c/x)(B/A)$, where $c$ is a fresh constant. 

$(\forall x D)(A) \dv    (\forall x  D)(B/A)$
follows from (IH), $E\forall$, and $I\forall$. 

\mm\item $C =\exists xD$. 
Left to the reader. 

\mm\item $C = \neg  D $.    

 (i)   $D = \forall x E $. 

 (IH) $(\neg E)(A)(c/x) \dv (\neg E)(B/A)(c/x)$, where $c$ is a fresh constant.


$(\neg\forall x E)(A) \dv  (\exists x \neg E)(A)$, by Proposition~\ref{prop.equiv.repl.qf}

 $ (\exists x \neg E)(A) \dv  (\exists x \neg E)(B/A)$, by (IH), $E\exists$, and $I\exists$ 

$(\exists x \neg E)(B/A) \dv  (\neg \forall x E)(B/A)$, by Proposition~\ref{prop.equiv.repl.qf}

(ii)  $D = \exists x E $. Left to the reader. 
 
\item $C = \con D $ 

(i) $D =\forall x E$. 

\m 
(IH)  $ E (A)(c/x) \dv   E (B/A)(c/x) $,  
$\neg E (A)(c/x) \dv \neg  E (B/A)(c/x) $,  \\ 
$\con E (A)(c/x) \dv \con E (B/A)(c/x) $, where $c$ is a fresh constant.  

\m 

$\con \forall x E(A) \dv  {\forall x (E(A)\land\con E(A)) \lor \exists x (\neg E(A)\land\con E(A))}$, 
by  Prop.~\ref{prop.equiv.repl.qf}

${\forall x (E(A)\land\con E(A)) \lor \exists x (\neg E(A)\land\con E(A))} \dv  \\ 
 \mbox{\hspace{13mm}} {\forall x (E(B/A)\land\con E(B/A)) \lor \exists x (\neg E(B/A)\land\con E(B/A))}$, \\
  \mbox{\hspace{13mm}} by (IH) and a few  derivation steps. 

$ {\forall x (E(B/A)\land\con E(B/A)) \lor \exists x (\neg E(B/A)\land\con E(B/A))} \dv   
 \con  \forall x E(B/A)$, by  Proposition~\ref{prop.equiv.repl.qf}.  

\m
(ii) $D =\exists xE$. Left to the reader.

 \end{enumerate}
\end{proof}

\end{theorem}

\subsection{Prenex normal form}

 A formula $A^{v}$ is an \textit{alphabetic variant} of a formula $A$ if they differ only in the names of some (or all) of their bound variables. It is well known that, in general, such formulas are equivalent in logics without recovery operators like $\con$, such as classical and intuitionistic logic.
  The proof is straightforward, since under a substitutional reading, such as the one adopted here, the truth conditions of the quantifiers depend on the constants that replace the variables, but not on the names of the variables themselves.
A syntactic proof depends only on repeated applications of the elimination and introduction rules for the quantifiers.

In several first-order logics equipped with $\con$, however, such a proof cannot be carried out, since it is not possible to remove the 
quantifiers from the scope of $\con$, 
and an explicit rule is therefore introduced to enforce such an equivalence (see e.g. \cite{qletf, qmbc, rod.ant.lu}).
In \qf, by contrast, such formulas can be proved equivalent, and this result essentially depends on Proposition~\ref{prop.equiv.repl.qf}, which allows the quantifiers to be removed from the scope of $\con$.

\begin{proposition}\label{prop.alphab.var} \ 

\m\mh
Let $A$ be a sentence in the language of \qf, and let $A^v$ be an \textit{alphabetic variant} of $A$.  
Then $A \dashv\vdash A^v$.
\begin{proof}
    In order to show that $Qx_1 Ax_1 \vdash Qx_2 Ax_2$ (where $Q$ is either $\forall$ or $\exists$), just apply the corresponding elimination and introduction rules.
The result then follows by applying the replacement property.
    \end{proof}

\end{proposition}

\begin{theorem}(Prenex normal form theorem) \label{th.prenex}

\m\mh A formula  $A$ of \qf\ is in \textit{prenex normal form} (PNF) if   
\begin{itemize}
     
    \item[]   (i) $A$ has the form
$Q_1x_1Q_2x_2\dots Q_n x_n C$, where each $Q_i$,  $0 \leq i \leq n$, is either $\forall $ or 
$\exists $ and $C$ is a  quantifier-free formula, or 

\item[]  (ii) $A$ is a generalized literal (cf.~Definition~\ref{def.literals}), or

\item[]  (iii) $A$ is a top particle, or $A$ is a bottom particle. 

\end{itemize}

\m\mh 

 \m\mh 
 For every formula $A \in Sent(\Lg)$ 
 there is a formula $B\in Sent(\Lg)$ in prenex normal form such that $A \dashv\vdash B$.



\begin{proof} The proof is by induction on the complexity $\cc$ of  $A$.

\m\mh If $\cc(A)=1$, then $A$ is an atom, and hence  it is in PNF.

\m\mh    
 If $\cc(A) > 1$ and $A$ is a generalized literal, a top particle, or a bottom particle, then $A$ is in PNF.
Otherwise, we proceed according to the following cases.

\begin{enumerate}[label=(\arabic*)]


\m
\item $A = B\lor C$. 
 
By (IH), there are formulas $B'$ and $C'$ in PNF
such that $B \dashv\vdash B'$ and $C\dashv\vdash C'$. 
Let $A'$ be the formula $B' \lor C'$.
Apply Proposition~\ref{prop.alphab.var} to rename the variables of $B'$ and $C'$ so as to obtain equivalent formulas
$B'' = Qx_1 Qx_2 \dots Qx_n D$ and $C'' = Qx_{n+1} Qx_{n+2} \dots Qx_m E$,
such that no variable of $B''$ occurs in $C''$, and vice versa.
Then apply items~(1) and~(2) of Proposition~\ref{prop.equiv.prenex} and replacement to obtain a formula in PNF equivalent to $A$ in \qletfp.

 \m\item  $A = B\land C$. Left to the reader.

\m\item $A=\forall x B$.


Consider the formula $B(c/x)$, obtained by applying $E\forall$ to $A$.  
By (IH),  there is a formula $B'(c/x)$ in PNF such that $B'(c/x) \dashv\vdash B(c/x)$.
 Now, applying $I\forall$ to $B'(c/x)$ gives $\forall x B'$. This formula is in PNF 
(recall that there are no void quantifiers in the language of \qf, see the beginning of Subsection~\ref{sectQLETF+}) and is equivalent to $A$ in \qletfp. 

 \m\item   $A=\exists x B$. Left to the reader.

\m\item $A= \neg  B$

\m  (i)  $B = \neg C$, so  $A = \neg\neg C$. 
 By (IH), there is a formula $C'$ in PNF such that $C' \dashv\vdash C$. Apply double negation and (IH).

\m (ii)  $B = C\land D$   

By (IH), there are formulas $C'$ and $D'$ in PNF such that $C' \dashv\vdash C$ and $D' \dashv\vdash D$.
By De Morgan and (IH), we obtain $\neg C' \lor \neg D'$.
Apply Proposition~\ref{prop.equiv.repl.qf} to remove quantifiers occurring within the scope of negations.
Then proceed as in item~(2) above, applying Proposition~\ref{prop.equiv.prenex}.


\m (iii)   $B = C\lor D$. Left to the reader.

\m (iv) $B =  \con C$.  

$A$ = $\neg\con C$.  
 By (IH), there is a formula $C'$ in PNF such that $C' \dv \con C$.
 Apply (IH) to obtain a formula $\neg C'$. 
 Then apply Proposition~\ref{prop.equiv.repl.qf} to remove quantifiers occurring within the scope of negation. By replacement, the resulting formula is equivalent to $A$ in \qletfp.

 \m (v) $B = \forall x C$. 

$A= \neg \forall x C$. By (IH), there is a formula $C'$ in PNF such that $C \dv\ C'$. 
Apply Proposition~\ref{prop.equiv.repl.qf}.

 \m (vi)   $B = \exists x C$. Left to the reader.

\item $A = \con B$

\m (i) $B = \neg C$. 

In this case, $A=\con \neg C$. 
(IH) there is a formula $D$ in PNF such that $D \dv  \neg C$. By replacement we obtain 
$\con D$, where $D$ is in PNF. $\con D$ has the form $\con Q E$, where  $Q$ is a quantifier and  $E$ is in PNF. 
By Proposition~\ref{prop.equiv.repl.qf}, from $\con Q E$ we obtain the formula 
$Q_1(E\land\con E) \lor Q_2(\neg E\land\con E)$, By (IH), there are formulas $E'$ and $E''$ in PNF such that 
$E'\dv \neg E$ and $E'' \dv \con E$. Apply (IH) to obtain the formula 
$Q_1(E\land E'') \lor Q_2( E'\land  E'')$. Then apply Proposition~\ref{prop.equiv.prenex} and replacement to obtain a formula in PNF equivalent to $A$ in \qletfp.


\m (ii)  $B =  C\land D$

$A= \con (C\land D)$. 
By Proposition~\ref{prop.equiv.con}, we obtain
$({\con C} \land {\con D} \land C \land D) \lor (\con C \land \neg C) \lor (\con D \land \neg D)$.
By (IH), there are formulas $C'$, $C''$, $C'''$, $D'$, $D''$, and $D'''$ in PNF such that
$C' \dashv\vdash C$, $C'' \dashv\vdash \neg C$, $C''' \dashv\vdash \con C$,
 $D' \dashv\vdash D$, $D'' \dashv\vdash \neg D$, and $D''' \dashv\vdash \con D$.
By replacement, we then obtain the formula
$({C'''} \land {D'''} \land C' \land D') \lor (C''' \land C'') \lor (D''' \land D'')$.
Finally, apply items~(1) and~(2) above to obtain a formula in PNF equivalent to $A$ in \qletfp, by replacement.

\m(iii)   $C\lor D$. Left to the reader.

\m(iv) $B=\con C$. In this case, $A=\con\con C$, which  is a top particle, so $A$ is in PNF. 

\m (v) $B=\forall x C$

In this case, $A = \con \forall x C$. Apply Proposition~\ref{prop.equiv.repl.qf} to obtain the formula
$\forall x (\con C \land C) \lor \exists x (\con C \land \neg C)$.
By (IH), there are formulas $C'$, $C''$, and $C'''$ in PNF such that
$C' \dashv\vdash C$, $C'' \dashv\vdash \neg C$, and $C''' \dashv\vdash \con C$.
By replacement, we then obtain the formula
$\forall x (C''' \land C') \lor \exists x (C''' \land C'')$.
Finally, proceed as in items~(1)-(4) above.

\m (vi)    $B=\exists x C$. 
Left to the reader
\end{enumerate}
\end{proof}

\end{theorem}

\section{Final remarks}   \label{sec.final.rem}

This paper introduced the logic of evidence and truth \qf. Like every \textit{LET}, \qf\ is an extension of \fde\ and is able to express six scenarios: the four scenarios of \fde\ plus two additional scenarios of reliable information.
\qf\ is proposed here as an information-based logic, capable of representing positive information and negative information -- which is the usual approach in information-based logics -- and also reliable information, expressed with the help of the unary operator \con. 
\qf\ admits a six-valued semantics, obtained from twist structures based on the bivalued semantics, and the six values can be understood as names of the six scenarios of the \lets.

\qf\ adopts the concept of an extended literal, which are formulas   $A$, $\neg A$, and $\con A$, for atomic $A$.
These formulas express the notions of, respectively, positive, negative, and reliable information, which are taken as primitive.
While $A$ and $\neg A$ have completely independent deductive behavior,  there are constraints concerning $\con A$: in the bivalued semantics it always receives value $0$ when the values of $A$ and $\neg A$ coincide, since in such cases there is clearly no reliable information about $A$.
In both the bivalued semantics and the six-valued semantics, once semantic values are assigned to extended literals, the semantic values of all formulas of the language are obtained. That is, such semantics are deterministic.

\qf\ assumes that the information that $A$ is reliable is itself reliable, hence the validity of $\con\con A$, and it is equipped with propagation rules, which are rules that transmit the operator \con\ from less complex formulas to more complex ones, and conversely.
These rules work as introduction and elimination rules for \con\ for complex formulas.
A non-classicality operator \incon\ is defined as $\incon A \defi \neg\con A$, and it has a deductive behavior dual to the operator \con, just as $\land$ and $\lor$ are dual to each other.

\qf\ starts from notions that are, so to speak, well-behaved and symmetric, which allowed it to be constructed in such a way that its logical operators are also  well-behaved and symmetric, both syntactically and semantically.
This made it possible for \qf\ to enjoy several desirable metatheoretical results, as shown in Sections~\ref{sec.properties.letfp} and~\ref{sec.prop.qf}, namely the conjunctive and disjunctive normal forms, the replacement property, and the prenex normal form theorem.

Some extensions of \qf\ naturally suggest themselves as worthwhile topics for further investigation.
The logic \letj, introduced in \cite{letj}, is equipped with a constructive implication, and its \con-free fragment is Nelson’s logic \nel. Quantified versions of \letj\ with constant and variable domains, dubbed \qcletj\ and \qvletj, have been studied in \cite{rod.ant.lu}. Extensions of these systems with propagation rules are likely to admit six-valued Kripke semantics. 
The logic \letkp, introduced in \cite{con.rod.sl}, whose implication-free fragment is the sentential logic \letfp, is an extension of \fdeto, obtained by adding   a material implication to \fde. We expect that \letkp\ can be extended to the first-order level with constant domains in a straightforward way, simply by adding a material implication to \qletfp.
These logics are certainly worth studying, but it is likely that the results obtained here with respect to \qletfp\ will not carry over to them.

 \bibliographystyle{plainnat}       
\bibliography{refs.main}

@book{troelstra.vandalen.1988,
	Author		= {A. S. Troelstra and D. {van Dalen}},
	Title		= {Constructivism in Mathematics}, 
	Volume		= {I},
	Year		= {1988}, 
	Publisher	= {North Holland},}

@article{gurevich.1977,
  author    = {Y. Gurevich},
  title     = {Intuitionistic Logic with Strong Negation},
  journal   = {Studia Logica},
  volume    = {36},
  pages     = {49--59},
  year      = {1977},
  doi       = {10.1007/BF02121114}
}

@article{qletf,
      title={Valuation semantics for first-order logics of evidence and truth}, 
      author={H. Antunes and A. Rodrigues and W. Carnielli and M. Coniglio},
      year={2022},
        Journal = {Journal of Philosophical Logic},
       volume = {51},
       number = {5},
       pages = {1141--1173},
        doi={10.1007/s10992-022-09662-8}, 
}

@Article{letf,
  author  = {Rodrigues, A. and Bueno-Soler, J. and Carnielli, W.},
  title   = {Measuring evidence: a probabilistic approach to an extension of {Belnap-Dunn logic}},
  journal = {Synthese},
  year    = {2020},
volume={198}, 
pages={5451-5480},
}

@article{carn:marcos:deamo:2000,
  title={Formal inconsistency and evolutionary databases},
  author={W. Carnielli and J. Marcos and S.  de Amo},
  journal={Logic and Logical Philosophy},
  volume={8},
  number={8},
  pages={115--152},
  year={2000}, 
}

@Article{fetzer.2004.dis,
  author  = {J. Fetzer},
  title   = {Disinformation: The Use of False Information},
  journal = {Minds and Machines},
  year    = {2004},
  volume  = {14},
  pages   = {231-240},
}

@Article{fallis,
  author  = {Fallis, D.},
  title   = {What is disinformation?},
  journal = {Library Trends},
  year    = {2015},
  volume  = {63},
  pages   = {401-426},
}

@Article{rod.ant.lu,
  author  = {A. Rodrigues and H. Antunes},
  journal = {Logica Universalis},
  title   = {First-order Logics of Evidence and Truth with Constant and Variable Domains},
  year    = {2022},
  doi     = {10.1007/s11787-022-00306-8},
}

@article{ShramkoWansing2019NatureEntailment,
  author  = {Shramko, Y. and Wansing, H.},
  title   = {The nature of entailment: an informational approach},
  journal = {Synthese},
  volume  = {198},
  number  = {Suppl. 22},
  pages   = {5241--5261},
  year    = {2021},
  doi     = {10.1007/s11229-019-02474-5}
}

@book{wardle2017,
  author      = {C. Wardle and H. Derakhshan},
  title       = {Information Disorder: Toward an Interdisciplinary Framework for Research and Policy},
  publisher = {Council of Europe},
  year        = {2017},
  url         = {https://rm.coe.int/information-disorder-report-2017/168076277c}
}

@Article{qmbc,
  author  = {W. Carnielli and M.E. Coniglio and R. Podiacki and T. Rodrigues},
  title   = {On the way to a wider model theory: completeness theorems for first-order logics of formal inconsistency},
  journal = {The Review Of Symbolic Logic},
  year    = {2014},
  volume  = {7},
  number  = {3},
  pages   = {548--578},
}

@InCollection{belnap.1977.how,
  author    = {Belnap, N. D.},
  title     = {How a computer should think},
  booktitle = {Contemporary {A}spects of {P}hilosophy},
  publisher = {Oriel Press},
  year      = {1977},
  editor    = {G. Ryle},
  note={Reprinted in  \textit{New Essays on Belnap-Dunn Logic}, Springer, 2019, pages 35-55},
}

@Book{wans.93,
  title     = {The Logic of Information Structures},
  publisher = {Springer},
  year      = {1993},
  author    = {H. Wansing},
}

@Article{dunn76,
  author  = {Dunn, J. M.},
  title   = {Intuitive semantics for first-degree entailments and `coupled trees'},
  journal = {Philosophical Studies},
  year    = {1976},
  volume  = {29},
  pages   = {149--168},
}

@Article{recovery,
  author  = {Carnielli, W. and Coniglio, M.E. and Rodrigues, A.},
  title   = {Recovery operators, paraconsistency and duality},
  journal = {Logic Journal of the IGPL},
  year    = {2019},
  volume  = {28},
  pages   = {624-656},
}

@Article{letj,
  author    = {Carnielli, W. and Rodrigues, A.},
  title     = {An epistemic approach to paraconsistency: a logic of evidence and truth},
  journal   = {Synthese},
  year      = {2017},
  volume    = {196},
  pages     = {3789-3813},
  doi       = {10.1007/s11229-017-1621-7},
  url       = {https://rdcu.be/ctJRQ},
}

@article{con.rod.sl,
  title   = {From {Belnap--Dunn} Four-Valued Logic to Six-Valued Logics of Evidence and Truth},
  author  = {M. E. Coniglio and A. Rodrigues},
  journal = {Studia Logica},
  year    = {2024},
  volume  = {112},
  number  = {3},
  pages   = {561--606},
  note    = {First online: 2023-08-09. Preprint: \url{https://bit.ly/47MZ61Z}}
}

@InCollection{wans.odin,
  author    = {Odintsov, S. and Wansing, H.},
  title     = {On the Methodology of Paraconsistent Logic},
  booktitle = {Logical Studies of Paraconsistent Reasoning in Science and Mathematics},
  publisher = {Springer},
  year      = {2016},
  editor    = {H. Andreas and P. Verd\'ee},
}

@article{can.M.fig.2020,
  title={On the logic that preserves degrees of truth associated to involutive {S}tone algebras},
  author={Cant\'u, M. and Figallo, M.},
  journal={Logic Journal of IGPL},
  year={2020},
  volume={28},
number = {5},
  pages={1000--1020}, 
}

@article{mar.riv.2022,
  title={Logics of involutive {S}tone algebras},
  author={Marcelino, S. and Rivieccio, U.},
  journal={Soft Computing},
  year={2022},
  volume={26},
number = {7},
  pages={3147--3160}, 
}

\end{document}